\documentclass[10pt]{amsart}
\usepackage{latexsym, amsmath, amssymb}
\usepackage{version}
\usepackage{epsfig,graphics}
\usepackage{color}
\usepackage{amsthm, amsopn, amsfonts}

\newtheorem{theorem}{Theorem}

\newtheorem{lemma}[theorem]{Lemma}
\newtheorem{corollary}[theorem]{Corollary}
\newtheorem{proposition}[theorem]{Proposition}
\newtheorem{definition}[theorem]{Definition}

\newcommand{\R}{{\mathbb R}}

\newcommand{\la}{\langle}
\newcommand{\ra}{\rangle}

\newcommand{\e}{\epsilon}

\renewcommand{\S}{{\mathbb S}}

\newcommand{\LE}{{\mathcal {LE}}}
\newcommand{\LEs}{{\mathcal {LE}^*}}
\begin{document}

\title
{
Local decay of waves on asymptotically flat stationary space-times 
}

\thanks{The author was supported in part by the NSF grant DMS-0801261.}

\author{Daniel Tataru}

 \address{Department of Mathematics, University of California,
   Berkeley, CA 94720-3840}

\begin{abstract}
In this article we study the pointwise decay properties 
of solutions to the wave equation on a class of 
stationary asymptotically flat  backgrounds in three space dimensions.
Under the assumption that uniform  energy bounds and 
a weak form of local energy decay hold forward in time
we establish a $t^{-3}$ local uniform decay rate for linear waves.
This work was motivated by open problems concerning decay rates 
for linear waves on Schwarzschild and Kerr backgrounds, where
such a decay rate has been conjectured by   Price~\cite{MR0376103}.
Our results apply to both of these cases.
\end{abstract}
% 
% 
% 
% \includeversion{1}
% \includeversion{2}

\maketitle

\section{Introduction}

In this article we consider the question of pointwise decay
for solutions to the wave equation on time independent, asymptotically
flat backgrounds. Our interest in this problem comes from
general relativity, more precisely the wave equation on Schwarzschild 
and Kerr backgrounds. There the conjectured decay, see Price~\cite{MR0376103}, is $t^{-3}$, and numerous partial results have been obtained.
Our main result applies to both the Schwarzschild and Kerr space-times, and 
proves a stronger form of Price's conjecture, where sharp bounds are
provided in a full forward light cone. However, the result 
does not depend on the specific geometry of either Schwarzschild and Kerr space-times; instead, it applies to a large class of stationary asymptotically flat space-times.

We use $(x,t=x_0)$ for the coordinates in $\R^{3+1}$. We use Latin 
indices $i,j=1,2,3$ for spatial summation and Greek indices
$\alpha,\beta=0,1,2,3$ for space-time summation. In $\R^3$ we also
use polar coordinates $x = r \omega$ with $\omega \in \S^2$. 
By $\la r \ra$ we denote a smooth radial function which agrees 
with $r$ for large $r$ and satisfies $\la r \ra \geq 2$. 
We consider a partition of $ \R^{3}$ into the  dyadic sets 
$A_k= \{\la r \ra \approx 2^k\}$ for $k \geq 0$.

All operators we consider in this paper have time independent
coefficients.  To describe their spatial size and regularity we use
the standard symbol classes $S(r^k)=S^{k}$:
\[
 a \in S(r^k) \Longleftrightarrow 
|\partial_x^j a(x)| \leq c_j \la r\ra^{k-j}, \quad j \geq 0
\]
By $S_{rad}(r^k)$ we denote spherically symmetric functions in $S(r^k)$.
We also introduce a slightly stronger class $l^1S(r^k)$ with 
some dyadic summability added,
\[
a \in l^1 S(r^k) \Longleftrightarrow \sum_m 2^{m(j-k)} \|\partial_x^j a(x)\|_{L^\infty(A_m)} \leq c_j, \quad j \geq 0
\]
 On occasion we allow for logarithmic modifications, and use the class of functions $S( \log r)$:
\[
 a \in S(\log r) \Longleftrightarrow |a| \lesssim \log \la r \ra, \quad |\partial_x^j a(x)| \leq c_j \la r\ra^{-j}, \quad j \geq 1
\]

We consider the pointwise decay problem for the forward wave equation
\begin{equation}
(\Box_g + V) u = 0, \qquad u(0) = u_0, \quad \partial_t u(0) = u_1
\label{ahom}\end{equation}
associated to d'Alembertian $\Box_g$ corresponding to 
to a Lorenzian metric $g$, and to a potential $V$. 
We work with potentials $V$ of the form 
\begin{equation}
V = V_{lr} + V_{sr}, \qquad V_{lr} \in S_{rad}(r^{-3}), \ \ 
V_{sr} \in l^1S(r^{-3})
\label{vassume}\end{equation}
For the metric $g$ we consider two cases:

\bigskip

{\bf Case A:} $g$ is a smooth Lorenzian metric  in $\R \times \R^3$, with the following properties:

(i) $g$ is time independent, i.e. $\partial_t$ is a Killing vector field.

(ii) The level sets $t = const$ are space-like.

(iii) $g$ is asymptotically flat in the following sense:
\[
 g = m + g_{sr} + g_{lr}, 
\]
where $m$ stands for the Minkowski metric, $g_{lr}$ is a long range 
spherically symmetric component, with $S_{rad}(r^{-1})$ coefficients,
of the form
\[
 g_{lr} = g_{lr, tt}(r) dt^2 + g_{lr, tr}(r) dt dr + g_{lr, rr}(r) dr^2 +  
g_{lr, \omega \omega}(r) r^2 d \omega^2
\]
and $g_{sr}$ is a short 
range component, with $l^1 S(r^{-1})$ coefficients, of the form
\[
 g_{sr} = g_{sr, tt} dt^2 + 2g_{sr, ti} dt dx_i + g_{sr, ij} dx_i dx_j 
\]

\bigskip

We remark that these assumptions guarantee that $\partial_t$
is time-like near spatial infinity, but not necessarily in a compact set.
This leads us to the second case we consider:

\bigskip

{\bf Case B:} $g$ is a smooth Lorenzian metric $g$ in  an exterior domain 
$\R \times \R^3\setminus B(0,R_0)$ which  satisfies
(i),(ii),(iii) above in its domain, and in addition

(iv) the lateral boundary $\R \times \partial B(0,R_0)$ is outgoing
space-like.

\bigskip

This latter condition insures that the corresponding wave equation is
well-posed forward in time. This assumption is satisfied in the case
of the Schwarzschild and Kerr metrics in suitable advanced time coordinates, such as the ones used in \cite{MMTT}, \cite{TT}. There the parameter $R_0$ is chosen so that $0 < R_0 < 2M$ in the case of the Schwarzschild metric, respectively
$r^- < R_0 < r^+$ in the case of Kerr (see \cite{MR0424186} for the definition of $r^{\pm}$), so that the exterior of the $R_0$ ball contains a
neighbourhood of the event horizon.

We consider the inhomogeneous forward Cauchy problem for the wave equation
with initial data at time $t=0$,
\begin{equation}
(\Box_g + V) u = f, \qquad u(0) = u_0, \quad \partial_t u(0) = u_1
\label{hom}\end{equation}
We denote the Cauchy data at time $t$ by $u[t] = (u(t),\partial_t u(t))$. 
In the sequel $H^k=H^k(\R^3)$ is the usual Sobolev space, and
by $\dot H^{k,1}$ we denote the functions $u$ with
$\nabla_x u \in H^k$. For the above evolution we begin with the following definition:

\begin{definition}
We say that the evolution \eqref{hom} is forward bounded if the following
estimates hold:
\begin{equation}
\| u[t]\|_{\dot H^{k,1} \times H^k} \leq c_k (\|u[0]\|_{\dot H^{k,1} \times H^k} +\|f\|_{L^1 H^k}),
\qquad t \geq 0, \quad k \geq 0 
\end{equation}
\end{definition}

It suffices to have this property when $f=0$. Then the $f$ term can be added in by the Duhamel formula.  The uniform forward bounds above
are easy to establish provided that the Killing vector field $\partial_t$
is everywhere time-like. Otherwise, there is no general result, but
various cases have been studied on a case by case basis. 

In the case of the Schwarzschild space-time, following  earlier work of Laba-Soffer \cite{MR1732864}, Blue-Soffer \cite{MR1972492}, \cite{MR2113761}, Twainy~\cite{Twainy}
and others, full uniform energy bounds were obtained by
Dafermos-Rodnianski \cite{MR2527808} in the exterior region. 
For a shorter alternate proof, as well as a discussion concerning the 
extension of the energy bounds to a neighbourhood of the event
horizon we refer the reader to the article \cite{MMTT} of the 
author and collaborators. 

As far as the Kerr space-time is concerned, partial results 
were first obtained by  Finster-Smoller~\cite{MR2405857},
and Finster-Kamran-Smoller-Yau~\cite{MR2215614}.
Full uniform energy bounds were proved in the the case of the Kerr space-time with small angular momentum in Dafermos-Rodnianski~\cite{DR3};
their result applies in effect to a larger class of small perturbations
of the Schwarzschild space-time. An alternate proof of the Kerr 
result, based on local energy decay estimates, was given shortly afterward in Tataru-Tohaneanu~\cite{TT}.

A stronger property of the wave flow is local energy decay. 
We introduce the local energy norm $LE$ by
\begin{equation}
 \| u\|_{LE} = \sup_k  \| \la r\ra^{-\frac12} u\|_{L^2 (\R \times A_k)} 
\label{ledef}\end{equation}
its $H^1$ counterpart,
\[
  \| u\|_{LE^1} = \| \nabla u\|_{LE} + \| \la r\ra^{-1} u\|_{LE}
\]
as well as the dual norm
\begin{equation}
 \| f\|_{LE^*} = \sum_k  \| \la r\ra^{\frac12} f\|_{L^2 (\R \times A_k)} 
\label{lesdef}\end{equation}
These definitions are specific to the three dimensional problems.
Some appropriate modifications are needed in other dimensions,
see for instance \cite{MT}.  

We also define similar norms for higher regularity,
\[
  \| u\|_{LE^{1,k}} = \sum_{|\alpha| \leq k} 
\| \partial^\alpha u\|_{LE^1}
\]
respectively 
\[
  \| f\|_{LE^{*,k}} =  \sum_{|\alpha| \leq k} 
\| \partial^\alpha f\|_{LE^{*}}
\]
The dual spatial norms
$\LE$ and $\LEs$ are the spatial counterparts of the space-time $LE$ and 
$LE^*$ norms,
\begin{equation}
\| v\|_{\LE} = \sup_k  \| \la r\ra^{-\frac12} v\|_{L^2 ( A_k)} 
\qquad 
 \| f\|_{\LE^*} = \sum_k  \| \la r\ra^{\frac12} f\|_{L^2 ( A_k)} 
\label{left}\end{equation}
Their higher regularity versions are defined as
\begin{equation}
\| v\|_{\LE^m} =   \sum_{j = 0}^m \sup_k  \| \la r\ra^{-\frac12} \partial_x^j v\|_{L^2 ( A_k)} 
\qquad 
 \| f\|_{\LE^{*,m}} = \sum_{j = 0}^m \sum_k  \| \la r\ra^{\frac12} \partial_x^j f\|_{L^2 ( A_k)} 
\label{leftk}\end{equation}
In Case A above this leads to the following

\begin{definition}
a) We say that the evolution \eqref{hom} has the local energy decay
property if the following estimate holds:
\begin{equation}
 \| u\|_{LE^{1,k}} \leq c_k (\|u[0]\|_{H^{1,k} \times H^k} + \|f\|_{LE^{*,k}}), \qquad k \geq 0
\label{le}\end{equation}
in $\R \times \R^3$.
\end{definition}

The first local energy decay estimates for the wave equation 
were proved in the work of Morawetz~\cite{MR0204828}, \cite{MR0234136},
\cite{MR0372432}. There is by now an extensive literature devoted 
to this topic and its applications; without being exhaustive we mention 
Strauss~\cite{Strauss}, Keel-Smith-Sogge~\cite{MR1771575},\cite{MR1945285},
Burq-Planchon-Stalker-Tahvildar-Zadeh~\cite{MR2106340},
 Metcalfe-Sogge~\cite{MR2217314},\cite{MR2299569}. 

The sharp form of the estimates as well as the notations above are
from Metcalfe-Tataru~\cite{MT};  this paper also contains a proof of the local energy decay estimates for small (time dependent) long range perturbations of the Minkowski space-time, and further references. See also the related  paper  Metcalfe-Tataru~\cite{MT1}, as well as  Alinhac's article~\cite{MR2266993}.

There is a related family of local energy estimates for the Schr\"odinger
equation. As a starting point we refer the reader to \cite{gS} and references therein.

Work in progress of Metcalfe-Tataru
suggests that in the nontrapping stationary case the local energy decay 
\eqref{le} is closely tied to a non-resonant
spectral behavior at frequency $0$ (see \cite{MMT} for our similar results
for Schr\"odinger evolutions).

In Case B an estimate such as \eqref{le} is unlikely to hold due to
the existence of trapped rays, i.e. null geodesics confined to a
compact spatial region. However a weaker form of the local energy decay
may still hold if the trapped null geodesic are hyperbolic. This is the
case for both the Schwarzschild metric  and for the
Kerr metric with angular momentum $|a| < M$.
To state such bounds we introduce a weaker version of the local energy
decay norm
\[
  \| u\|_{LE^1_{w}} = \| (1-\chi) \nabla u\|_{LE} + \| \la r\ra^{-1} u\|_{LE}
\]
for some spatial cutoff function $\chi$ which is smooth and compactly supported. Heuristically,
$\chi$ is chosen so that it equals $1$ in a neighbourhood of the trapped
set. We define as well  a dual type norm
\[
 \| f\|_{LE^*_w} = \| \chi f\|_{L^2 H^1}+ \sum_k 
 \| \la r\ra^{\frac12} f\|_{L^2 (A_k)} 
\]
As before we define the higher norms $LE^{1,k}_{w}$ respectively
$LE^{*,k}_w$.

\begin{definition}
 We say that the evolution \eqref{hom} has the weak local energy decay
property if the following estimate holds:
\begin{equation}
 \| u\|_{LE^{1,k}_w} \leq c_k (\|u[0]\|_{\dot H^{1,k} \times H^k} + \|f\|_{LE^{*,k}_w} ), \qquad k \geq 0
\label{lew}\end{equation}
in either $\R \times \R^3$ or
in the exterior domain case.
\label{d:weakle}\end{definition}

Two examples where weak local energy decay is known to hold are the 
Schwarzschild space-time, and the Kerr space-time with  
small angular momentum 
$|a| \ll M$.

The estimates in the Schwarzschild case are the end result of a 
series of papers by Laba-Soffer \cite{MR1732864}, Blue-Soffer \cite{MR1972492}, \cite{MR2113761},\cite{BS}, Blue-Sterbenz~\cite{MR2259204} Dafermos-Rodnianski~\cite{MR2527808},~\cite{DR1} and Marzuola-Metcalfe-Tataru-Tohaneanu~\cite{MMTT}. 
In particular we note the contribution of \cite{MR2527808}
to the understanding of the red shift effect near 
the event horizon, as well as that of \cite{BS}, \cite{MMTT}
concerning improved bounds near the photon sphere 
(i.e. the trapped set). The sharpest form of the estimates
is obtained in  \cite{MMTT}; 
in particular the bounds there fit into the framework
of Definition~\ref{d:weakle}.

In the case of the Kerr space-time with small angular momentum 
$|a| \ll M$ the local energy estimates were first proved in Tataru-Tohaneanu~\cite{TT}, in a form which is compatible 
with Definition~\ref{d:weakle}. A more detailed presentation as well
as stronger bounds near the trapped set are contained 
in Tohaneanu's Ph.D Thesis~\cite{TMihai} as well as in the follow-up
paper ~\cite{Mihai}. For related subsequent work we also refer the reader
to Dafermos-Rodnianski~\cite{DR4} and Andersson-Blue~\cite{AB}.

The high frequency analysis of the dynamics near the hyperbolic trapped 
orbits has a life of its own, but not much to do with the present article.
For more information we refer the reader to work of de Verdi{\`e}re-Parisse~\cite{MR1294470}, Christianson~\cite{MR2450154},
Nonnenmacher-Zworski~\cite{NZ} and references therein.

We remark that the local energy decay property is stronger than the mere boundedness of the evolution. This is not as clear in the case of the weak local energy decay; however, one case which is well understood is when the
Killing vector field $\partial_t$ is time-like within the support of
the cutoff function $\chi$ in the definition of $LE_w^1$.

In this article we consider the pointwise decay of solutions to the
wave equation \eqref{hom}.  If $g$ is the Minkowski space-time then, by
Huygens principle, if the initial data has compact support then $u$
will vanish in a compact set after some time. If instead we add a stationary potential which has exponential decay
at infinity and no eigenvalues or zero resonances, then one obtains exponential decay of
solutions inside any compact set. For polynomially decaying
stationary perturbations of the Minkowski space-time the question of the precise decay rates is still largely open. In principle, one expects that 
these decay rates will be determined by three factors: (i) the spectral behavior near time frequency $0$ (e.g. eigenvalues, resonances), (ii) the trapping properties of the  metric and (iii) the behavior of the metric at spatial infinity.

This work was inspired by an idea in Tataru~\cite{gS}
and Metcalfe-Tataru~\cite{MT}, which is roughly that the local energy estimates contain all the important local and low frequency information concerning the flow, and that only leaves the analysis near spatial infinity to be understood. Precisely, in \cite{gS} and \cite{MT} it is proved that local energy decay implies Strichartz estimates in the asymptotically 
flat setting, first for the Schr\"odinger equation and then for the
wave equation. 
The same idea was exploited in \cite{MMTT} and \cite{Mihai}
to prove Strichartz estimates for the wave equation  on the Schwarzschild and then on the Kerr space-time.

The result we obtain in this article  shows that the above philosophy 
applies as well to the pointwise local decay. 
 Precisely, we prove that if  weak local energy decay estimates 
hold, then the pointwise decay rate of the solutions 
to the wave equation is only determined by the  
behavior of the perturbation at spatial infinity. 

To define the initial data spaces we use several vector fields:
the generators of spatial translations $T = \{ \partial_1,\partial_2,\partial_3\}$,   the generators of rotations
$ \Omega = \{ x_i \partial_j - x_j \partial_i \}$ and the spatial scaling
$S_r = r \partial_r$. Then we define the weighted Sobolev spaces
$Z^{m,k}$ with norms
\begin{equation}
 \| f\|_{Z^{m,n}} = \sup_{i+j+k \leq m} \| \la r \ra^n  
T^i \Omega^j S_r^k f\|_{\LE^*} 
\label{zdef}\end{equation}
Now we are ready to state the main result of this article:

\begin{theorem} \label{maint}
Let $m$ be a large enough integer. 
Let $g$ be a metric which satisfies the conditions (i), (ii), (iii)
in $\R \times \R^3$, or  (i), (ii), (iii), (iv)
in $\R \times \R^3\setminus B(0,R_0)$. Let $V$ be a potential as in \eqref{vassume}. Assume that the 
evolution \eqref{hom} is forward bounded and satisfies 
the weak local energy decay estimates \eqref{lew}. 
Then in normalized coordinates the solution $u$ to the homogeneous problem
\eqref{ahom} satisfies the bounds
\begin{equation}
|u(t,x)| \lesssim \frac{1}{\la t+|x|\ra \la t-|x| \ra^2}
\left(\|u_0\|_{ Z^{m+1,1}} + \|u_1\|_{Z^{m,2}}\right)
\label{pd}\end{equation}
respectively
\begin{equation}
|\partial_t u(t,x)| \lesssim \frac{1}{\la t+|x|\ra \la t-|x| \ra^3}
\left(\|u_0\|_{ Z^{m+1,1}} + \|u_1\|_{Z^{m,2}}\right)
\label{pdt}\end{equation}

\end{theorem}

The normalized coordinates referred to in the theorem are introduced in
the next section via a spherically symmetric change of coordinates.
These are needed in order to insure that the outgoing light cones are close
to Minkowski cones. 

The required decay of the initial data at infinity is roughly a
summable $r^{-3}$ for the position $u_0$, respectively a summable
$r^{-4}$ for the velocity $u_1$. These are sharp. By contrast, we made
no effort to optimize the value of $m$.

In view of the local energy decay estimates in
Marzuola-Metcalfe-Tataru-Tohaneanu \cite{MMTT}, respectively
Tataru-Tohaneanu~\cite{TT}, the hypothesis of the theorem is valid in
the case of the Schwarzschild space time, and also the Kerr space-time
with small angular momentum, with $V=0$. Thus we obtain:

\begin{corollary}[Price's Law] The decay estimates \eqref{pd},
  respectively \eqref{pdt} hold for the Schwarzschild space-time, as
  well as for the Kerr space-time with small angular momentum, with
  respect to coordinates which coincide with the Regge-Wheeler
  coordinates near spatial infinity and with the Eddington-Finkelstein
  coordinates near the event horizon.
\end{corollary}

This problem has had a long history. Dispersive $L^1 \to L^\infty$
estimates providing $t^{-1}$ decay of $3+1$ dimensional waves in the
Minkowski setting have been known for a long time. There is an
extensive literature concerning dispersive estimates for $\Box + V$;
this is by now a well understood problem, for which we refer the
reader to the survey paper of Schlag~\cite{MR2333215}.

The need for weighted decay inside the cone arose in John's proof
\cite{MR535704} of the Strauss conjecture in $3+1$ dimensions.
Decay bounds for $\Box+V$, similar to those
in the theorem,  were obtained by Strauss-Tsutaya~\cite{MR1432072}
and Szpak~\cite{MR2475479}, \cite{Szpak}, see also
Szpak-Bizo{\'n}-Chmaj-Rostworowski~\cite{MR2512504}.

A more robust way of proving pointwise estimates via $L^2$ bounds and
Sobolev inequalities was introduced in the work of Klainerman, who
developed the so-called vector field method, see for instance
\cite{MR865359}. This idea turned out to have a myriad of
applications, and played a key role in the
Christodoulou-Klainerman~\cite{MR1316662} proof of the asymptotic
stability of the Minkowski space time for the vacuum Einstein
equations.

In the context of the Schwarzschild space-time, Price conjectured the
$t^{-3}$ decay rate for linear waves. More precise heuristic
computations were carried out later by
Ching-Leung-Suen-Young~\cite{PhysRevLett.74.2414},
\cite{PhysRevLett.74.4588}. Following work of Wald~\cite{MR534342},
the first rigorous proof of the boundedness of the solutions to the
wave equation was given in Kay-Wald~\cite{MR895907}.

Uniform pointwise $t^{-1}$ decay estimates
were obtained by Blue-Sterbenz~\cite{BS} and also Dafermos-Rodnianski~\cite{MR2527808};
the bounds in the latter paper are stronger in that they extend 
uniformly up to the event horizon. Very recently, a local $t^{-\frac32}$
decay result was obtained by Luk~\cite{L}. These results are obtained
using multiplier techniques, related to Klainerman's vector field 
method; in particular the conformal multiplier plays a key role.

Another venue which was explored was to use the spherical symmetry
in order to produce an expansion into spherical modes, and to study the
corresponding ode. This was pursued by Kronthaler~\cite{MR2226325}, \cite{K},
who in the latter article was able to establish the sharp Price's Law 
in the spherically symmetric case. A related analysis was carried 
out later by Donninger-Schlag-Soffer~\cite{DSS} for all the spherical
modes;  they  establish a $t^{-2}$ local decay for each spherical mode.

Switching to Kerr, the first decay results there were obtained 
by Finster-Kamran-Smoller-Yau \cite{MR2215614}. Later Dafermos-Rodnianski~\cite{DR4}
were able to extend their Schwarzschild results to Kerr, obtaining
almost a $t^{-1}$ decay. 

Finally, we mention the related work of
Dafermos-Rodnianski~\cite{MR2199010}, where Price's Law is established
in a nonlinear but spherically symmetric context.

We conclude the introduction with a brief overview of the paper.  In
the next section we discuss the normalization of the coordinate
system, which is done in order to insure that the outgoing null cones
are close to Minkowski cones. In the following section we use the
uniform forward energy bounds in order to define resolvent operators
$R_\tau$ associated to the evolution \eqref{ahom} in the lower
half-space $\{\Im \tau < 0\}$, and to establish straightforward $L^2$
and $H^m$ bounds for the resolvent.  These bounds degenerate as $\Im
\tau$ approaches $0$, so they provide no information about the
resolvent on the real axis.

In Section~\ref{s:res-le} we transfer the local energy decay estimates
to the resolvent $R_\tau$. These bounds have the key advantage that
they are uniform as $\tau$ approaches the real axis.

Using a version of the vector field method, in Proposition~\ref{putau}
we commute the resolvent with the rotations $\Omega$ and scaling $S$
to obtain a larger class of resolvent bounds. These bounds allow us in
Proposition~\ref{prext} to prove that we can extend the resolvent
$R_\tau$ to the real axis continuously in a weaker topology as well as
show that it inherits all the estimates previously proved in the lower
half space. We also establish the key outgoing radiation condition
\eqref{outrad} for $R_\tau$ for real $\tau$.

The limit $R_0$ of the resolvent $R_\tau$ as $\tau \to 0$ is
considered in Propositions~\ref{pu00}, \ref{pu0}; this allows us to
define $R_0$, and establish some key properties for it. The precise
description of $R_0$ is used in Proposition~\ref{pusmall} to produce
an expansion up to second order for the resolvent $R_\tau$ near
$\tau=0$.

In Section~\ref{linf} we switch from $L^2$ to $L^\infty$ bounds in the
resolvent estimates for real $\tau$, in a way which is reminiscent of
the Klainerman-Sobolev embeddings. Exploiting the outgoing radiation
condition \eqref{outrad}, we also obtain improved decay estimates at
spatial infinity for $(\partial_r -i\tau)^m R_\tau$.

In the last section we conclude the argument. Assuming for simplicity
that $u_0=0$, the time Fourier transform of the solution $u$ to the
forward wave equation \eqref{ahom} is represented as $\hat u = R_\tau
u_1$.  Thus in order to invert the Fourier transform and obtain
pointwise estimates for $u$ we need to control the size and regularity
with respect to $\tau$ of $R_\tau u_1$.  This is relatively easy for
large $\tau$ because we allow the initial data to have high Sobolev
regularity in Theorem~\ref{maint}. The delicate issue is the resolvent
behavior at $\tau = 0$. In order to obtain a $t^{-3}$ decay for $u$ at
infinity, we use the second order expansion for $R_\tau$ near $\tau =
0$ given by Proposition~\ref{pusmall}.  In particular we are able to
identify the leading contribution of the long range part of the wave
operator and of the potential.

\section{The coordinate normalization}

Here we carry out a preliminary step which brings the metric to a
canonical form near spatial infinity and removes certain logarithmic
type corrections in the resolvent bounds later on. This is achieved in
two steps, a spherically symmetric coordinate change followed by a
conjugation with respect to a scalar weight.  This transformation is
roughly equivalent to using Regge-Wheeler coordinates in
Schwarzschild/Kerr near spatial infinity.

Our final goal is to reduce the problem to the case when the operator
$\Box_g+V$ in \eqref{ahom} is replaced by an $L^2$ selfadjoint operator $P$
which has the form
\begin{equation}
P = -\partial_t^2 + \Delta_x   + P_{lr}  + P_{sr}
\label{can-p}\end{equation}
where the long range spherically symmetric part has the form
\begin{equation}
 P_{lr} = g_{lr}^{\omega}(r) \Delta_{\omega} + V_{lr}(r), \qquad g_{lr}^\omega, V_{lr} \in S_{rad}(r^{-3})
\label{can-p-lr}\end{equation}
and the selfadjoint short range part $P_{sr}$ has the form
\begin{equation}
 P_{sr} = \partial_\alpha g_{sr}^{\alpha \beta}\partial_\beta + V_{sr},
\qquad g_{sr}^{\alpha \beta} \in l^1 S(r^{-1}), \quad V_{sr} \in l^1S(r^{-3})
\label{can-p-sr}\end{equation}
with $g^{00}_{sr} = 0$.

\subsection{ Normalized coordinates}

We first remark that the condition (iii) translates into a similar
representation for the dual coefficients $g^{\alpha\beta}$,
\[
g^{\alpha\beta} = m^{\alpha \beta} + g^{\alpha\beta}_{lr} + g^{\alpha\beta}_{sr}
\]
We can simplify these coefficients via a radial 
change of coordinates as follows:

\begin{lemma} 
  There exists a spherically symmetric change of coordinates so that
  the long range component $g_{lr}^{\alpha \beta}$ of
  $g^{\alpha\beta}$ has the additional properties
\begin{equation}
 g_{lr}^{tt} = -g_{lr}^{rr}, \qquad g_{lr}^{rt} = 0
\label{hrrtt}\end{equation}
for large $r$.
\label{coord}\end{lemma}
We will refer to the coordinates given by this lemma as {\em normalized coordinates}. These are the coordinates used in Theorem~\ref{maint}.
In Schwarzschild or Kerr space-times these coordinates are precisely 
the Regge-Wheeler coordinates for large $r$. In a bounded spatial region,
any choice of coordinates which is nonsingular on the event horizon 
will suffice.
 
\begin{proof}
To achieve the relation $g_{lr}^{rt} = 0$ we keep $r$ unchanged and 
reset $t$ by setting
\[
 t:= t+ a(r)
\]
where $a$ is determined by
\[
 a'(r) = -\frac{g_{lr}^{rt}}{1+g_{lr}^{rr}}
\]
Then we have $a' \in S_{rad}(r^{-1})$ and $a \in S_{rad}(\ln r)$, therefore the property (iii) is left unchanged. We remark that this can only be done 
provided that $1+g_{lr}^{rr} > 0$. This is clearly true for large $r$,
and can be arranged for small $r$ by modifying the short range part $P_{sr}$.

Once $g_{lr}^{rt}=0$, we insure that $g_{lr}^{rr} = g_{lr}^{tt}$ with a radial change of coordinates 
\[
dr^* = \left(\frac{1-g_{lr}^{tt}}{1+g_{lr}^{rr}}\right)^\frac12 dr
\]
This gives $r^* - r \in S_{rad}(\ln r)$ and $dr^*/dr \in 1+S_{rad}(r^{-1})$.
Thus while both transformations above leave condition (iii)
unchanged, both the functions $a(r)$ and $r^*(r)-r$ may have 
a logarithmic component as $r \to \infty$. The proof is concluded.

\end{proof}

\subsection{Left and right multiplication by scalar functions}

Consider a metric $g$ as in \eqref{hrrtt}. The corresponding d'Alembertian
has the form
\[
\Box_g = \frac{1}{\sqrt g} \partial_\alpha \sqrt g \, g^{\alpha \beta} \partial_\beta
\]
where it is easy to see that $g \in 1 + S_{rad}(r^{-1})+ l^1S(r^{-1})$. This is turned into an $L^2$ selfadjoint operator via the conjugation
\[
\Box_g \to g^{\frac14} \Box_g g^{-\frac14}
\]
To further insure that the coefficient $g^{tt}_{lr}$
of $\partial_t^2$ is set to $-1$ we multiply the above conjugated operator by 
$(-g^{tt})^{-1} \in 1+l^1S(r^{-1})$; in order to preserve the $L^2$ selfadjointness, half of it will go on the left and half will go on the right. 

Thus, our  final replacement for $\Box_g+V$ is the operator
\[
P = g^{\frac14} (-g^{tt})^{-\frac12} (\Box_g +V)
(-g^{tt})^{-\frac12} g^{-\frac14}
\]
Commuting, $P$ can be written in divergence form, and it 
will have the required form given by \eqref{can-p}, \eqref{can-p-lr}
and \eqref{can-p-sr}.

Also it is easy to see that such a transformation leaves 
unchanged the three properties of the forward evolution \eqref{hom}:
boundedness, local energy decay and pointwise decay.

\section{The resolvent and energy estimates} 

 Consider the solution $u$ to the forward homogeneous problem
\begin{equation}
 P u = 0, \qquad u(0) = 0, \ u_t(0) = g \in L^2,
\label{firsteq}\end{equation}
extended by $0$ to negative times.

By the uniform energy bounds its time Fourier transform $\hat u(\tau)$
is a $\dot H^1$ valued distribution in $\R$, and admits a holomorphic
extension into the lower half plane. We define the {\em resolvent}
operator $R_\tau$ by 
\[
 R_\tau g = \hat u(\tau), \qquad \Im \tau < 0
\]
The uniform energy bounds \eqref{le} for $P$ translate into
an $L^2$ bound for $R_\tau$,
\begin{equation}
 \|R_\tau g\|_{\dot H^1} + \tau \| R_\tau g\|_{L^2}
\lesssim \frac{1}{|\Im \tau|} \|g\|_{L^2}, 
\qquad \Im \tau < 0,
\label{energyres}\end{equation}
as well as similar bounds for higher Sobolev norms,
\begin{equation}
 \| R_\tau g\|_{\dot H^{k,1}} + \tau \| R_\tau g \|_{H^k}
\lesssim \frac{1}{|\Im \tau|} \|g\|_{H^k}, 
\qquad \Im \tau < 0.
\label{energyreshigh}\end{equation}

In order to write an equation for the resolvent we
express the operator $P$ in the form
\[
P =  -\partial^2_t + P^1\partial_t + P^2
\]
where, in view of  \eqref{can-p}, \eqref{can-p-lr}
and \eqref{can-p-sr}, $P^1$ is a short range spatial skew-adjoint 
operator and $P_2$ a second order
spatial operator of the form
\[
P^2 = \Delta + P_{lr} + P^2_{sr} 
\]
with both long and short range components. 

Since $u$ is extended by $0$ to negative times,
it solves the distributional equation 
\[
P u = g \delta_{t=0}
\]
Taking a Fourier transform in this relation, it
follows that $\hat u(\tau)$ solves the  equation
\begin{equation}
 P_\tau \hat u(\tau) = g, \qquad  P_\tau = \tau^2 -i \tau P^1 + P^2
\label{ptau}\end{equation}
Hence for the resolvent we obtain the equation
\begin{equation}
v = R_\tau g \implies P_\tau v = g 
\label{resolvent}\end{equation}
By Duhamel's formula we  also obtain a representation for the Fourier 
transform of the solution $v$ to the full forward Cauchy problem
\begin{equation}
 P v = f, \qquad v(0) = h, \ v_t(0) = g
\label{secondeq}\end{equation}

\begin{lemma}
 Let $v$ be the solution to \eqref{secondeq} with 
$f \in L^\infty L^2$, $h \in \dot H^1$ and $g \in L^2$.
Then for $\Im\tau < 0$ we have
\begin{equation}
 \hat v(\tau) = R_\tau( \hat f(\tau) + \tau h + 2 P_1 h + g)
\label{cpr}\end{equation}
\label{l:cp}\end{lemma}

This allows us to show that $P_\tau$ and $R_\tau$ are inverse 
operators if $\Im \tau < 0$:

\begin{lemma} Let $\Im \tau < 0$. Then $P_\tau: H^2 \to L^2$
is a bounded one to one operator with dense range,
 and $R_\tau$ is its inverse.
\end{lemma}
\begin{proof}
Given $u_0 \in H^2$, apply  the previous lemma to the function 
$ u = u_0 1_{t \geq 0}$ to see that $R_\tau P_\tau u_0 = u_0$.
The range is dense since it contains $H^1$. Indeed, by \eqref{resolvent} and \eqref{energyreshigh}, for $f \in H^1$ we have $R_\tau f \in H^2$ and $P_\tau R_\tau f = f$.
\end{proof}

\section{The resolvent and local energy decay} 
\label{s:res-le}

The goal of this section is to repeat the arguments of the 
previous section, but using the local energy decay property instead 
of the uniform boundedness. For this we consider the solution $u$ to the inhomogeneous forward problem
\begin{equation}
 Pu = g \in LE^{*,k+4}, \qquad k \geq 0
\label{inhom}\end{equation}
with $g$ supported away from $t = -\infty$.  From the local energy
decay estimate \eqref{lew} we obtain the weaker bound
\begin{equation}
\| u\|_{LE^{1,k}} \leq c_k \| f\|_{LE^{*,k+3}}
\label{lewa}\end{equation}
with a loss of three derivatives. We will use this bound to derive
a similar resolvent bound.
For each $\tau$ we use the $\LE^k$ norms defined in \eqref{left}, 
\eqref{leftk} to introduce the $\tau$ dependent norm $\LE_\tau$ by
\begin{equation}
\| v\|_{\LE_\tau^k} =  \|(|\tau| +\la r\ra^{-1})  u \|_{\LE^k}  + \|\nabla u\|_{\LE^k}  +  \|(|\tau| +\la r\ra^{-1})^{-1}  \nabla^2 u \|_{\LE^k} 
\label{letau}\end{equation}
 Then we have 

\begin{proposition}
 Let $\Im \tau < 0$, $k \geq 0$, $g \in \LE^{*,k+4}$ and $v = R_\tau g$. 
If the weak local energy decay bound \eqref{lewa} holds then we 
also have
\begin{equation}
\| v\|_{\LE_\tau^k} 
 \leq c_k  \|g\|_{\LE^{*,k+4}}
\label{lesmall}\end{equation}
\label{l:lesmall}
\end{proposition}
While this estimate is for now stated for $\Im \tau < 0$, it is
nevertheless uniform as $\tau$ approaches the real axis.
Later we will use this fact to extend this estimate to all real $\tau$.
 If $- \Im \tau \gtrsim 1$ then this bound is superseded by \eqref{energyreshigh}. Thus only the case $-1 < \Im \tau < 0$ is of interest. 
\begin{proof}
If we consider $f$ supported in a fixed dyadic region $A_l$ and 
measure $u$ in another region $A_m$, from \eqref{lewa}
we obtain the uniform dyadic estimates
\[
 \sum_{i \leq k} \ \ 2^{-\frac{m}2}\|\nabla_{x,t}^{i+1} u \|_{L^2(\R \times A_m)}
+ 2^{-\frac{3m}2} \| \nabla_{x,t}^{i} u \|_{L^2 (\R \times A_m)}
 \lesssim \sum_{i \leq k+3} 2^{\frac{l}{2}} \| \nabla_{x,t}^i f\|_{L^2 (\R \times A_l)}
\]
Since the equation \eqref{inhom} is solved forward in time,
it is straightforward to add an exponentially decreasing weight
in the above estimate. Precisely, from the above estimate one obtains a weighted version
\[
\begin{split}
 \sum_{i \leq k}  2^{-\frac{m}2}\|e^{-\epsilon t} \nabla_{x,t}^{i+1} u \|_{L^2(\R \times A_m)}
\! + \! 2^{-\frac{3m}2} \|e^{-\epsilon t} \nabla_{x,t}^{i} u \|_{L^2 (\R \times A_m)}
\!\! \lesssim \!\! \sum_{i \leq k+3} \!\! 2^{\frac{l}{2}} \| e^{-\epsilon t} \nabla_{x,t}^i f\|_{L^2 (\R \times A_l)}
\end{split}
\]
which holds uniformly with respect to $\epsilon > 0$. If we 
further restrict $\epsilon$ to $0 < \epsilon < 1$ then we can uniformly 
move the exponential inside the differentiation to obtain 
\[
 \sum_{i \leq k}  2^{-\frac{m}2}\| \nabla_{x,t}^{i+1}e^{-\epsilon t} u \|_{L^2(\R \times A_m)}
+ 2^{-\frac{3m}2} \| \nabla_{x,t}^{i} e^{-\epsilon t} u \|_{L^2 (\R \times A_m)}
\!\!\lesssim \!\!\!\sum_{i \leq k+3} 2^{\frac{l}{2}} \|  \nabla_{x,t}^i e^{-\epsilon t} f\|_{L^2 (\R \times A_l)}
\]
Since we only have Hilbert space norms in the above estimate, we can 
take a Fourier transform and use Plancherel's theorem to obtain
\[
\begin{split}
2^{-\frac{m}2} \sum_{i \leq k} (1+|\tau|)^{k-i}\left(   \| \nabla_{x}^{i+1} \hat u(\tau-i\epsilon) \|_{L^2(\R \times A_m)}
+ (2^{-m}+ |\tau|) \| \nabla_{x}^{i} \hat u (\tau -i\epsilon) \|_{L^2 (\R \times A_m)} \right)
\\ \lesssim \sum_{i \leq k+3} 2^{\frac{l}{2}} (1+|\tau|)^{k-i}
\|  \nabla_{x}^i  \hat f(\tau -i\epsilon)\|_{L^2 (\R \times A_l)}
\end{split}
\]
where the $L^2$ norms are now taken with respect to $\tau \in \R$
and $x \in \R^3$. 

We recall that by \eqref{cpr} we have 
$\hat u (\tau -i\epsilon) = R_\tau \hat f(\tau -i\epsilon)$.
Observing that the above bound holds for all $f$ in a dense 
subset of the Hilbert space defined by the norm on the right,
it follows that in effect its pointwise version with respect to $\tau$
must also be valid, namely 
\[
\begin{split}
2^{-\frac{m}2} \sum_{i \leq k} (1+|\tau|)^{k-i}\left(   \| \nabla_{x,t}^{i+1}  v \|_{L^2( A_m)}
+ (2^{-m}+|\Re \tau|) \| \nabla_{x,t}^{i}  v \|_{L^2 ( A_m)} \right)
\\ \lesssim \sum_{i \leq k+3} 2^{\frac{l}{2}} (1+|\tau|)^{k-i}
\|  \nabla_{x,t}^i   g\|_{L^2 ( A_l)}
\end{split}
\]
Here we have redenoted $\tau-i\epsilon$ by $\tau$, which now 
satisfies $\Im \tau < 0$, and replaced $\hat f(\tau)$ by $g$ and
$\hat u(\tau)$ by $v = R_\tau g$. We remark that in view of \eqref{energyreshigh} we can replace the factor $|\Re \tau|$ by 
$|\tau|$ in the second term on the left.

Taking the suppremum over all $m \geq 0$, and then summing over 
all $l \geq 0$, we arrive at
\begin{equation}
\begin{split}
 \sum_{i \leq k} (1+|\tau|)^{k-i} \left(\| \nabla_x v \|_{\LE^{i}}
+  \|(\la r \ra^{-1}+|\tau|) v \|_{\LE^i} \right)
 \lesssim \sum_{i \leq k+3}  (1+|\tau|)^{k+3-i}
\| g\|_{\LE^{*,i}}
\end{split}
\label{locft}\end{equation}
This concludes the transition from local energy decay 
to resolvent bounds. We remark that at the formal level
one can go directly from \eqref{lewa} to \eqref{locft}
for real $\tau$ via Plancherel's theorem. However, in order to rigorously 
justify this analysis we are using the intermediate step
of deriving uniform bounds for $\Im \tau < 0$.

It remains to prove that \eqref{locft} implies \eqref{lesmall}.
This depends on the size of $\tau$. If $|\tau| \approx 1$ there
is nothing to do. In the two remaining cases this transition
is carried out via simple elliptic arguments.

(a) The case $|\tau| \ll 1$. Then we need to estimate the second derivative of $v$ for large $r$. We write the equation for $v$ in the form
\[
P^2 v = -\tau^2 v + i \tau P^1 v
\]
This allows us to estimate
\[
\| (\la r \ra^{-1} + |\tau|)^{-1} P_2 v\|_{\LE^k} \lesssim
\tau \| v\|_{\LE^k} + \| \nabla_x v\|_{\LE^k} + \|g\|_{\LE^{*,k}}
\]
where all the right hand side terms are already controlled 
by the right hand side in \eqref{lewa} via \eqref{locft}. 
Since $P^2$ is elliptic for large $r$, the transition from
$P_2 v$ bounds to $\nabla^2 v$ bounds is done in a standard elliptic fashion within each dyadic region $A_m$.

(b) The case $|\tau| \gg 1$. The left hand side of \eqref{locft}
already  contains the $\LE^k$ norm of $v$, we only need to
be able to discard the $\tau$ factors on the right. For this we
observe that the equation for $v$, namely
\[
(\tau^2 + i \tau P^1 + P^2) v = g
\]
is elliptic for low spatial frequencies $|\xi| \ll \tau$.
This suggests we split $g$ into a low frequency and a high frequency part,
\[
g = S_{\leq \tau} (D_x) g + S_{\geq \tau} (D_x) g := g_{low} + g_{high}
\]
For the high frequency part we can estimate 
\[
\sum_{i \leq k+3}  (1+|\tau|)^{k+3-i}
\| g_{high} \|_{\LE^{*,i}} \lesssim \|g\|_{\LE^{*,k+3}}
\]
and use \eqref{locft} directly.

For the low frequency part, on the other hand, we reiterate
the equation, setting 
\[
v_{low} = \tau^{-2} g_{low} + v_1
\]
where $v_1$ solves 
\[
P_\tau v_1 = i \tau^{-1} P^1 g_{low} + \tau^{-2} P^2 g_{low} = g_{low,1}
\]
Thus we have traded powers of $\tau$ for derivatives, which 
is favourable since $g_{low}$ is low frequency. Reiterating 
several times we arrive at
\[
v_{low} = \sum_{j=0}^{k-1} \tau^{-2-j} Q_j g_{low} + w_{low} 
\]
where $Q_j$ are partial differential operators of order $j$ 
with smooth bounded coefficients and $w_{low}$ solves
\[
P_\tau w_{low} = \tau^{-k} Q_k g_{low} + \tau^{-k-1} Q_{k+1} g_{low} := h_{low}
\]
The first term in $v_{low}$ is estimated directly in terms of $g_{low}$
while for $w_{low}$ we remark that 
\[
\sum_{i \leq k+3}  (1+|\tau|)^{k+3-i}
\| h_{low} \|_{\LE^{*,i}} \lesssim \|g\|_{\LE^{*,k+3}}
\]
and use \eqref{locft}. The proof of the proposition is complete.

\end{proof}

\section{Vector field bounds and the resolvent limit on the real axis}

Our first goal in this section is to prove further estimates for the operator $P_\tau$ by  commuting it with several vector fields:

i) The rotations $\Omega=\{
x_i \partial_j-x_j \partial_i\}$. Commuting them with $P_\tau$ we
obtain 
\[
 [P_\tau,\Omega]= Q_{sr}
\]
where $Q_{sr}$ stands for a short range operator of the form
\[
 Q_{sr} = \tau (h^{0i} \partial_i + \partial_i h^{0i}) + \partial_i h^{ij}
\partial_j + h, \qquad h^{0i}, h^{ij} \in l^1S(r^{-1} ), \qquad h \in l^1 S(r^{-3}) 
\]

ii) The scaling $S = -\tau \partial_\tau + r \partial_r$. Then the commutator is 
\[
 [P_\tau,S]= 2 P_\tau + Q_{lr}+ Q_{sr}
\]
where the long range component is radial and has the form 
\[
 Q_{lr} = k^\omega \Delta_{\omega} + k, \qquad k^\omega, k \in S_{rad}(r^{-3})
\]

The next proposition applies for spatial functions $g$, in which case 
$S = S_r g$. However, it is also interesting to allow $g$ to depend
on the parameter $\tau$.

\begin{proposition}
Let $\Im \tau < 0$ and $g \in \LE^*$, possibly depending 
on $\tau$, so that 
\begin{equation}
\| T^i \Omega^j S^k g\|_{\LE^*} \leq 1, \qquad i+4j+16k \leq N
\label{oslarge}\end{equation}
Then
\begin{equation}
\| T^i \Omega^j S^k R_\tau g\|_{\LE_\tau} \lesssim 1,
\qquad i+4j+16k \leq N-4
\label{ossmall}\end{equation}
\label{putau}\end{proposition}

\begin{proof}
Set $v = R_\tau g$. To illustrate the method we first consider the simplest cases.
 In the case of $\Omega$ we have
\[
 P_\tau \Omega v = \Omega g + Q_{sr} v
\]
A direct computation shows that
\begin{equation}
 \|Q_{sr} v\|_{\LE^{*,k}} \lesssim   \| v\|_{\LE_\tau^{k}}
\label{qsr}\end{equation}
Hence by \eqref{lesmall} we obtain
\[
\| \Omega v\|_{\LE_\tau^k} \lesssim \|\Omega g \|_{\LE^{*,k+4}}
+ \|g\|_{\LE^{*,k+8}}
\]

Consider now the operator $\Omega^2$. We have
\[
 P_\tau \Omega^2 v = \Omega^2 f + Q_{sr} \Omega v + Q_{sr} v
\]
so we conclude again via \eqref{qsr} and \eqref{lesmall} that
\[
 \| \Omega^2 v\|_{\LE_\tau^k} \lesssim 
\|\Omega^2 g \|_{\LE^{*,k+4}}+
\|\Omega g \|_{\LE^{*,k+8}}
+ \|g\|_{\LE^{*,k+12}}
\]

Next consider the scaling field $S$.
We have 
\[
 P_\tau S v = S g + 2 g + Q_{sr} v + Q_{lr} v
\]
For $Q_{sr}$ we use \eqref{qsr}, while for $Q_{lr}$ we use
\begin{equation}
 \|Q_{lr} v\|_{\LE^{*,k}} \lesssim   \| v\|_{\LE_\tau^k} + \| \Omega^2 v\|_{\LE_\tau^k} 
\label{qlr}\end{equation}
to obtain
\[
 \| Sv\|_{\LE_\tau^k} \lesssim \|S g \|_{\LE^{*,k+4}}+ 
\|\Omega^2 g \|_{\LE^{*,k+8}}+
\|\Omega g \|_{\LE^{*,k+12}}
+ \|g\|_{\LE^{*,k+16}}
\]
The proof of the general case is by induction. The details are left
for the reader.
\end{proof}

Now we have enough information in order to define 
the  pointwise limit of the resolvent as $\tau$ 
approaches the real line. The next result deals with the limit 
on $\R \setminus\{ 0\}$. The detailed analysis near frequency zero 
is left for the next section.

\begin{proposition}
i) The operators $R_\tau$ extend continuously from the lower half-space
to the real line $\R\setminus\{ 0\}$ in the $H^4_{comp} \to L^2_{loc}$ topology, and the bound \eqref{lesmall} holds uniformly for all real $\tau$. 

ii) Let $\tau \in \R \setminus \{0\}$ and $g \in \LE^{*,4}$. Then the function 
$v = R_\tau g$ satisfies the outgoing radiation condition
\begin{equation}
\lim_{j \to \infty} 2^{-\frac{j}2} \| (\partial_r +i\tau) v\|_{L^2(A_j)} = 0
\label{outrad}\end{equation}

iii) Conversely, suppose that $\tau \in \R \setminus 0$, and $v \in \LE_\tau^4$ which satisfies the outgoing radiation condition \eqref{outrad}. If
$P_\tau v = g \in \LE^{*,4}$ then $v=R_\tau g$.

\label{prext}\end{proposition}

This result allows us to transfer the previous bounds 
in this section from the lower half-plane to the real axis.  
\begin{corollary}
The results in Propositions~\ref{l:lesmall},\ref{putau} and   Proposition~\ref{pusmall} apply for all nonzero real $\tau$.
In addition, for $i,j,k$ as in Proposition~\ref{putau}, the functions
$T^i \Omega^j S^k R_\tau g$ satisfy the outgoing radiation condition
\eqref{outrad}. 
\end{corollary}

\begin{proof}
i) In view of the uniform bound \eqref{lesmall}
it suffices to establish the convergence on a dense subset.
Let $g \in \LE^*$ so that \eqref{oslarge} holds for some 
large $N$. Then 
by \eqref{ossmall} we have 
\[
 \| S R_\tau g\|_{\LE_\tau} \lesssim 1
\]
Hence we can write
\[
 \tau \partial_\tau R_\tau g = S R_\tau g + r \partial_r R_\tau
\]
In view of the $\LE_\tau$ bounds for $R_\tau g$ and $S R_\tau g$
we can bound the LHS locally in $L^2$ and in $H^1$.
Hence the map $\tau \to R_\tau g$ is locally Lipschitz away from 
$\tau =0$ from $\{ \Im \tau \leq 0\}$ into $H^1_{loc}$.  The desired 
convergence follows.

ii)  By \eqref{lesmall} it  suffices to establish the radiation condition \eqref{outrad} for $g$ in a dense subset of $\LE^{*,4}$.
Here  we consider functions $g \in \LE^*$ so that \eqref{oslarge} holds 
for $g$ for some large $N$, and so that in addition $g$ decays 
one order faster at infinity, $\la r \ra g \in \LE^{*}$.
We denote $v(\tau) = R_\tau g$. 

For some large $R$  we  truncate $v(\tau)$ to the 
exterior of a ball of size $R$ and set 
\[
w(\tau) = \chi_{> R} v(\tau).
\]
This does not affect its behavior at spatial infinity 
and thus the radiation condition.
The functions $w(\tau)$ solve an equation of the form
\begin{equation}
 (\Delta + \tau^2) w = Q_{lr} w + Q_{sr} w  
+ [ P_\tau,\chi_{> R}] v(\tau)+
\chi_{> R} g:= h(\tau)
\label{flatpert}\end{equation}
We can now use the local energy decay estimate for $\Delta + \tau^2$,
namely
\begin{equation}
 \| \nabla w \|_{\LE} 
+ \|(|\tau|+\la r\ra^{-1}) w\|_{\LE}  \lesssim \|h\|_{LE^*}
\label{lebox}\end{equation}
which can be either proved directly or derived from 
\eqref{le} for $\Box$ in the same way \eqref{lesmall} was obtained
from \eqref{lew}.

In view of \eqref{ossmall} and of the condition $\la r \ra g \in \LE^{*}$
we can estimate uniformly 
for $\tau$ in a compact subset of $\{ \Im \tau \leq 0\} \setminus\{ 0\}$ 
\[
\| \la r \ra h(\tau)\|_{\LE^*} \lesssim 1.
\]
The radiation condition follows if we can show that the solutions 
$w(\tau)$ to \eqref{flatpert} satisfy 
\begin{equation}
 \| \la r \ra (\partial_r + i\tau) w \|_{\LE} \lesssim
 \|\la r \ra h\|_{\LE^*} 
\label{radcheck} \end{equation}
uniformly on compact sets. 
Indeed, let us estimate $w$ in a dyadic region  $A_m$.
We split $h = \chi_{< m-2} h + \chi_{\geq m-2} h$.
We have 
\[
 \|\chi_{\geq m-2} h\|_{\LE^*} \lesssim 2^{-m} \|\la r \ra h\|_{\LE^*}
\]
therefore by \eqref{lebox} we have
\[
 \| \chi_m  (\partial_r + i\tau) (\Delta + \tau^2)^{-1}
\chi_{\geq m-2} h \|_{\LE} \lesssim 2^{-m}  \|\la r \ra h\|_{\LE^*}
\]

It remains to prove the same bound for the other component $\chi_{< m-2} h$. For this use the kernel for the 
operator $(\Delta + \tau^2)^{-1}$ which is 
\[
K_\tau(y) = |y|^{-1} e^{-i\tau |y|}
\]
This is justified for $\Im \tau < 0$, where the 
symbol of $\Delta + \tau^2$ is nonzero, and in effect
bounded from below. A direct computation gives
\[
 (\partial_r + i\tau) (\Delta + \tau^2)^{-1} h(x) 
= e^{-i\tau |x|} \int h(y) |x-y|^{-1} e^{-i\tau |x-y|} a(x,y) dy
\]
where 
\[
 a(x,y) = -\frac{x \cdot (x-y)}{|x||x-y|^2} 
- i \tau \left( \frac{x \cdot (x-y)}{|x||x-y|} -1 \right)
\]
We are interested in the case where we have the localization
$|x| \approx 2^{m}$ and $2|y| \leq |x|$. There we have 
$|a| \lesssim 2^{-m} \la y \ra $. Taking absolute values 
inside the integral, we use again \eqref{lebox} with $\tau = 0$
to obtain
\[
 \| \chi_m  (\partial_r + i\tau) (\Delta + \tau^2)^{-1}
\chi_{< m-2} h \|_{\LE} \lesssim 2^{-m}.  
\]
The proof of \eqref{radcheck} is concluded. 

iii) If $v$ is compactly supported then for $\epsilon > 0$ we have
\[
v = R_{\tau - i\epsilon} P_{\tau - i \epsilon} v 
\]
But the right hand side converges to $R_\tau g$ locally in $L^2$,
and the proof is concluded.

Suppose now that $v$ is supported away from $0$. 
We first use the radiation condition for $v$ to derive 
a similar radiation condition for its derivatives,
\begin{equation}
\lim_{j \to \infty} 2^{-\frac{j}2} \| \nabla^k (\partial_r +i\tau) v\|_{L^2(A_j)} = 0, \qquad k \leq 4
\label{outdf}\end{equation}
We recall that  $P_\tau$ has the form 
\begin{equation}
P_\tau = \Delta+\tau^2+ P_{lr} + P_{sr},
\label{ptauexp}\end{equation}
therefore by commuting and estimating all terms with decaying coefficients
in terms of $v$ we obtain
\[
\| \la r \ra^{-1} (\Delta+ \tau^2) (\partial_r + i \tau) v\|_{\LE^{*,4}}
\lesssim \| v\|_{\LE_\tau^4} + \|g\|_{\LE^{*,4}}.
\]
Hence from the radiation condition \eqref{outrad} it follows that
\[
\lim_{j \to \infty} 2^{-\frac{j}2} \| \Delta (\partial_r +i\tau) v\|_{L^2(A_j)} = 0.
\]
Combined with \eqref{outrad} and elliptic estimates for $\Delta$ we obtain a similar decay property for $\nabla (\partial_r +i\tau)$ and 
$\nabla^2 (\partial_r +i\tau)$. Reiterating we arrive at \eqref{outdf}.

For $\epsilon > 0$ we define the functions $v_\epsilon = v e^{-\epsilon \la r\ra}
$
which satisfy
$v_\epsilon = R_{\tau-i\epsilon} P_{\tau-i\epsilon} v_\epsilon$.
By \eqref{lesmall} and part (i) of the proposition, the
 desired conclusion would follow if we are able 
to show that
\begin{equation}
P_{\tau-i\epsilon} v_\epsilon \to g \qquad \text{ in } \LE^{*,4}
\label{kyt}\end{equation}
For this we use \eqref{ptauexp} to compute directly
\[
P_{\tau-i\epsilon} v_\epsilon = g e^{-\epsilon \la r\ra}
+ 2 \epsilon (\partial_r + i \tau + \frac{1}r) v e^{-\epsilon \la r\ra} 
+ (\epsilon \ l^1S(r^{-2}) + \epsilon^2 l^1S(r^{-1})+ \epsilon\ l^1 S(r^{-1})  \nabla)
 v e^{-\epsilon \la r\ra} 
\]
Since $g \in \LE^{*,4}$, the first term converges to $g$ in 
$\LE^{*,4}$. For the $(\partial_r + i\tau) v$ part of the second term,
the same follows from the expanded radiation condition \eqref{outdf}.
Finally for the remaining terms we obtain  $O(\epsilon)$ decay in $\LE^{*,4}$
from the $\LE^4_\tau$ bound for $v$. Thus \eqref{kyt} follows and the 
proof is concluded.
\end{proof}

\section{ The resolvent analysis near zero frequency}

In this section we first study the limit of the resolvent at frequency 
$\tau=0$. Then we consider data with more regularity and decay
at spatial infinity, and we produce an expansion of the zero resolvent
in terms of powers of $r$. Finally, we obtain a quadratic expansion of the resolvent near frequency $0$ for data with faster decay.
We begin by establishing the existence of the limit of the resolvent at zero.

\begin{proposition} [Resolvent extension to $\tau=0$] \label{pu00}
a) Let $g, \la r\ra g \in \LE^{*,4}$. Then the limit
\[
 R_0 g = \lim_{\epsilon \to 0} R_{-i\epsilon} g
\]
exists in the $L^2_{loc}$ topology, and the following bound holds:
\begin{equation}
 \| \la r\ra R_0 g\|_{\LE_0^k} \lesssim  \|\la r \ra  g\|_{\LE^{*,4+k}},
\qquad k \geq 0
\label{r0a}\end{equation}
b) The operator $R_0$ admits a unique continuous extension to $\LE^{*,4}$
which satisfies
\begin{equation}
 \|  R_0 g\|_{\LE_0^k} \lesssim  \|g\|_{\LE^{*,4+k}}, \qquad 
k \geq 0
\label{r0b}\end{equation}
respectively
\begin{equation}
\lim_{j \to \infty} 2^{(m-\frac{3}2)j} \| \nabla^m R_0 g\|_{L^2(A_j)} = 0,
\qquad j = 0,1,2
\label{zerorad} \end{equation}

c) Conversely, let $v \in \LE_0^{4}$ so that $P_0 v = g \in \LE^{*,4}$
and 
\begin{equation}
\lim_{j \to \infty} 2^{-\frac{j}2} \| \partial_r v\|_{L^2(A_j)} = 0.
\label{zeroradsuf} \end{equation}
Then $v = R_0 g$.
\end{proposition}

\begin{proof}
{\bf (a), the case $k=0$}. Denote $v_\epsilon = R_{-i\epsilon} g$.
By the local energy decay estimate \eqref{lesmall} we know that 
\[
\| v_\epsilon\|_{\LE_{i\epsilon}} \lesssim 1
\]
This already suffices in order to obtain the existence 
of a strong limit $v_\epsilon \to v$ in $H^1_{loc}$ on a subsequence.
To prove convergence it suffices to establish the uniqueness
of the limit. For this we first study the regularity of the limit.
A-priori we know that $v \in \LE_0$, but this does not suffice
for uniqueness.

The equation for $v_\epsilon$ has the form $P_{-i\epsilon} v_\epsilon = g$,
or in expanded form
\[
 (P^2 - \epsilon P^1 - \epsilon^2) v_\epsilon = g
\]
with $P^2$ selfadjoint and $P^1$ skew-adjoint.
Since $v_\epsilon \in H^2$,  we cam multiply
by $v_\epsilon$ and integrate by parts to obtain
\[
 \Re \int  g v_\epsilon dx = \int -P^2 v_\epsilon \cdot v_\epsilon+\epsilon^2  v_\epsilon^2 dx   
\]
The principal part of $- P_2$ is positive definite outside 
a compact set $K$, so we obtain
\[
 \|\nabla v_\e\|_{L^2}^2 \lesssim \left|\int  g v_\epsilon dx\right|
+ \int_K |\nabla v_\epsilon|^2 dx 
 + \int  \la r\ra^{-3} |v_\epsilon|^2 dx 
\]
For the first term on the right we use the fact that
$\la r\ra g \in \LE^*$ and $\la r\ra^{-1} v_\epsilon \in \LE$. For the second term on the right we use the $\LE_{i\epsilon}$ bound for $v_{\epsilon}$.
In the last term we use the Hardy inequality to absorb the
far away part into the left hand side, and the $\LE_{i\epsilon}$ norm for the near part. We obtain an uniform bound 
\[
\|\nabla v_\e\|_{L^2}^2 \lesssim \|\la r\ra g\|_{\LE^{*,4}}
\]
so any limit $v$ of a subsequence must belong to $\dot H^1$.
Suppose we have two distinct limit points $v_1$, $v_2$. The
difference $v= v_1-v_2$ must satisfy $v \in \dot H^1$, $P_0 v = 0$.
But this would contradict the local energy decay estimates for
the function $v \ 1_{t \geq 0}$ which solves the homogeneous wave equation in $[0,\infty)$. 
 
It remains to boost the decay of $v$ and its derivatives near infinity
to $\la r \ra v \in \LE_0$.
For this we view $P_0$ as a perturbation of $-\Delta$ near infinity.
Precisely, for large $R$ we consider the equation for $w=\chi_{>R} v$,
\[
 P_0 w = \chi_{>R}(r) g + [P_0,\chi_{>R}(r)]v:= h
\]
Since the above commutator has compact support, we can estimate
\[
 \|\la r \ra  h\|_{\LE^{*}} \lesssim 
\|\la r \ra  g \|_{\LE^{*}} + \|v\|_{\LE_0} \lesssim \|g\|_{\LE^{*,4}}
\]
We rewrite the above equation as
\begin{equation}
 -\Delta w = h + \chi_{> R/2}(r) (P_{lr}+P_{sr}) w
\label{nearinf}\end{equation}
We solve this last equation via a Picard iteration.
It is easy to verify directly that
\[
\|\la r \ra  (-\Delta)^{-1} h \|_{\LE_0}
 \lesssim \|\la r \ra  h \|_{\LE^*}
\]
where $(-\Delta)^{-1}$ is the convolution operator with $r^{-1}$.

On the other hand the last term on the right is perturbative for large $R$,
\[
\|\la r \ra  \chi_{> R/2}(r) (P_{lr}+P_sr) w\|_{\LE^*} \lesssim R^{-1} 
\|\la r \ra  w \|_{\LE_0}
\]
Thus a unique solution $w \in \la r \ra^{-1} \LE_0 \subset \dot H^1$  is found iteratively. On the other hand, by energy estimates as above 
the equation \eqref{nearinf} also has a unique solution in 
$\dot H^1$, therefore the $\la r \ra^{-1} \LE_0 $ coincides with 
the $\dot H^1$ solution constructed from $v$.
  This proves the existence of the limit of $u_\epsilon$ as $\epsilon \to 0$, as well as 
\eqref{r0a} for $k=0$.

{\bf (b), proof of \eqref{r0b}.} The bound \eqref{r0b} is obtained by passing to the limit $\tau = i \epsilon \to 0$ in the estimate \eqref{lesmall}. 
Since $\la r\ra^{-1} \LE^{*,4}$ is dense in $ \LE^{*,4}$, this provides
us with an unique continuous extension which satisfies the bound 
\eqref{r0b}.

{\bf (a), the case $k > 0$.} Now \eqref{r0a} for $k \geq 1$ follows by combining the regularity of $v$ 
provided by \eqref{r0b} for small $r$ with the decay at infinity 
given by the $k=0$ case of \eqref{r0a} plus elliptic regularity
for large $r$. The reason we need to argue in this roundabout way
is that $P_0$ is elliptic for large $r$, but not necessarily for 
small $r$.

{\bf (b), proof of \eqref{zerorad}.} By \eqref{r0b}, it suffices to establish \eqref{zerorad} for $g$ in a dense subset of $\LE^{*,4+k}$.
But this is provided by \eqref{r0a}.

{\bf (c).} This proof is identical to the proof of Proposition~\ref{prext} (iii). First the higher regularity version of \eqref{zeroradsuf} are
established, namely
\[
\lim_{j \to \infty} 2^{-\frac{j}2} \| \nabla^k \partial_r v\|_{L^2(A_j)} = 0, \qquad k \leq 4.
\]
Then the relation $v = P_0 g$ is obtained as the limit of
$v_\epsilon = R_{-i\epsilon} P_{-i\epsilon} v_\epsilon$
for $v_\epsilon= v e^{-\epsilon r}$. 
\end{proof}

Our next result provides better bounds for $R_0$ on functions 
with more regularity and decay.

\begin{proposition}[Higher regularity for $R_0$] 

% a) If $g \in Z^{k+4,0}$ then
% \begin{equation}
%  \| R_0 g\|_{Z^{k,-2}} \lesssim  \| g\|_{Z^{k+4,0}}
% \label{r0fc}\end{equation}
a)  If  $g \in Z^{k+4,1}$ then the following 
representation holds for large $r$ 
\begin{equation}
R_0 g = c(r) \la r\ra^{-1} + v_1
\label{r0d}\end{equation}
where, with $N$ arbitrarily large, we have the bounds
\begin{equation}
\|c\|_{L^\infty}
+ \|S_r c\|_{l^1 S(1)} +\|v_1\|_{Z^{k,-1}} \lesssim \|g\|_{Z^{k+4,1}}
\label{r0da}\end{equation}
b) If we further have $g \in Z^{k+4,2}$ then  
 $R_0 f$ admits the large $r$ representation
\begin{equation}
R_0 g = c \la r\ra^{-1} + d(r) \cdot \nabla \la r\ra^{-1} + e(r) \la r\ra^{-2} + v_{2}
\label{r0e}\end{equation}
where, with $N$ arbitrarily large, we have the bounds
\begin{equation}
|c| + \|d\|_{L^\infty} + \| S_r d\|_{l^1S(1)} + \|e\|_{S(1)}
 + \|v_{2}\|_{Z^{k,0}} 
\lesssim  \| g\|_{ Z^{k+4,2}}
\label{r0f}\end{equation}
\label{pu0}\end{proposition} 
Here, by a slight abuse of notation, by the  $S(1)$ and $l^1S(1)$ 
norm we understand the sum of an arbitrarily large but finite 
number of seminorms for the respective symbol classes.

  We note the key presence of the $e(r) r^{-2}$ term in \eqref{r0e},
which is the highest order contribution due to the long range scalar
potential. This term will later turn out to produce the first 
nonanalytic contribution in the expansion of $R_\tau$ with respect to $\tau$ near $\tau = 0$. Thus this term ultimately determines
the final decay rate in our main theorem.

\begin{proof}
% a) The estimate \eqref{r0fc} for $R_0 g$ follows directly 
% from \eqref{r0b} for small $r$. If $r$ is large then we combine
% \eqref{r0b} with elliptic regularity bounds for the operator $P_0$.

  a) Consider \eqref{r0d}. For the local (small $r$) regularity of $v$
  we only need the bound \eqref{lesmall}. Thus it suffices to prove
  the same estimate for the solution $w$ to the near infinity problem
  \eqref{nearinf}. We begin with the simpler case $h \in \la r
  \ra^{-1} \LE^*$, in which case we seek to prove that $w$ admits a
  representation as in \eqref{r0d} so that
\begin{equation}
 \|c\|_{L^\infty} + \| S_r c\|_{l^1S(1)} +
\sum_{j \leq 2} \|\la r \ra^{-1+j} \nabla_x^j v_1\|_{\LE^*} 
\lesssim \|\la r \ra h\|_{\LE^*}
\label{fitr}\end{equation}
The second right hand side term in \eqref{nearinf} is perturbative with respect to this estimate with a small $O(R^{-1})$ bound, so it suffices to prove that this representation
is valid for the simpler
equation
\[
-\Delta w = h
\]
Then $w$ is computed using the fundamental solution for $-\Delta$ in $\R^3$, namely 
\[
w =  h \ast r^{-1}
\]
We consider a dyadic decomposition of $h$ with respect to the 
$\{A_m\}$ partition of $\R^3$,
\[
h = \sum_m h_m
\]
Correspondingly we partition $w$ into
\[
w = \sum_m w_m^{low} + w_m^{high}, \qquad w_m^{low} = \chi_{<m+2}\ ( h_m  \ast r^{-1})
\]
For the first part we directly estimate in a scale invariant 
 elliptic fashion
\[
2^{-2m} \|w_m^{low}\|_{L^2} + 2^{-m} \|\nabla w_m^{low}\|_{L^2}+ \|\nabla^2 w_m^{low}\|_{L^2} \lesssim \|h_m\|_{L^2}
\]
Hence is easy to see that the sum of the $w_m^{low}$ contributions
can be included in $v_1$.

Consider now the functions $w_m^{high}$, which we expand 
near infinity as
\begin{equation}
w_m^{high} = c_m \chi_{> m+2}  |x|^{-1}  + z_m,
\label{fgk}\end{equation}
where
\[
 c_m = \int h_m(y) dy, \qquad z_m= \chi_{> m+2}  \int h_m(y)  \left(\frac{1}{|x-y|} - \frac{1}{|x|}\right) dy.
\]
By Cauchy-Schwarz we can estimate 
\[
\| c_m\|_{l^1} \lesssim \|\la r \ra  h\|_{\LE^*}
\]
therefore the first term in \eqref{fgk} yields the first term in 
\eqref{r0d}, with 
\[
c(r) = \sum_m c_m \chi_{> m+2} 
\]
The second term in \eqref{fgk}, on the other hand,
 can be incorporated into $v_1$. Indeed we can bound
\[
\left|\frac{1}{|x-y|} - \frac{1}{|x|}\right| \lesssim \frac{|y|}{|x|^2},
\qquad |x| > 2 |y|
\]
which yields the off-diagonal decay
\[
\|\la r \ra^j \nabla^j z_m\|_{L^\infty(A_n)} \lesssim 2^{m-n}, \qquad 
m \geq n+2 
\]
which suffices after summation with respect to $n$ and $m$.
This concludes the proof of \eqref{fitr}. The more general bound 
\eqref{r0da} follows from \eqref{fitr} by elliptic regularity
once we observe that $ -\Delta (c(r) r^{-1}) \in l^1 S(r^{-2})$.

b) Finally we consider the case when $g \in Z^{k+4,2}$.
Again it suffices to prove the same estimate for the solution 
$w$ to the near infinity problem \eqref{nearinf}.
In a manner similar to the previous case, we will first construct a decomposition \eqref{r0e} of $w$ so that the following estimate holds:
\begin{equation}
|c| + \|d\|_{L^\infty} + \| S_r d\|_{l^1S(1)} + R^{-\frac12} \|e\|_{S(1)}  + \sum_{j \leq 2} \|\la r \ra^{2-j} \nabla_x^j v_2\|_{\LE^*} 
 \lesssim \|h \|_{\la r\ra^{-2} \LE^*+ R^{\frac12} S_{rad}(r^{-4})}
\label{fdr}\end{equation}
Once this is done, its higher regularity version \eqref{r0f} 
follows by elliptic theory. 

We remark that in the above estimate we have harmlessly added an $S_{rad}(r^{-4})$ component to $h$. Its purpose is to 
render the second right hand side term in \eqref{nearinf} 
fully perturbative. The $R^{\pm \frac12}$ weights serve the same goal.
To understand how this works we split $w$ into  $w= c \la r \ra^{-1}
+ w_1$, and consider all possible terms:

i) The output of both the short range operator $P_{sr}$
and the long range operator $P_{lr}$ applied to  $w_1$ can be estimated 
as follows:
\[
 \| \chi_{> R} (P_{sr}+P_{lr}) w_1\|_{\la r\ra^{-2} \LE^*} \lesssim R^{-1}
(\|d\|_{L^\infty} + \| S_r d\|_{l^1S(1)} +  \|e\|_{S(1)}  + \sum_{j \leq 2} \|\la r \ra^{2-j} \nabla_x^j v_2\|_{\LE^*} )
\]
which provides an $R^{-\frac12}$ gain compared with \eqref{fdr}.

ii) The output of the short range operator $P_{sr}$ applied to 
$c(r) \la r \ra^{-1}$ satisfies the weaker bound 
\[
 \| P_{sr} c \la r \ra^{-1}\|_{l^1 S(r^{-4})} 
 \lesssim  |c|,
\]
which gives decay when restricted to $r \gtrsim R$,
\[
 \| P_{sr} c \la r \ra^{-1}\|_{\la r\ra^{-2} \LE^*} 
 \lesssim  o_R(1) |c|,
\]
where the decay of the constant as $R \to \infty$ comes from the 
fact that the $l^1$ summation is restricted to dyadic regions 
$A_m$ with $2^m \gtrsim R$.

iii) The output of the $k^\omega(r) \Delta_\omega$
part of $P_{lr}$ applied to 
$c(r) \la r \ra^{-1}$ is zero.

iv) Finally he output of the $k(r)$
part of $P_{lr}$ applied to $c(r) \la r \ra^{-1}$ belongs to $S(r^{-4})$,
and we have an $R^{-\frac12}$ gain compared to \eqref{fdr} due to the choice of constants in \eqref{fdr}.

 Thus we need to prove the bound \eqref{fdr} for an expansion 
\eqref{r0f} of the solution $w$ to the equation
\[
-\Delta w  = h, \qquad \text{supp } h \subset \{ r \gtrsim R\}
\]
We consider first the case when $\la r \ra^2 h \in \LE^*$. 
We proceed as in the proof of \eqref{r0d}-\eqref{r0da}.
The analysis of $w_m^{low}$ rests unchanged. However, 
for $w_m^{high}$ we need a second order expansion at infinity,
namely 
\begin{equation}
w_m^{high} = c_m \chi_{> m+2}  |x|^{-1}  +  d_m^i  \chi_{> m+2}
\partial_i r^{-1} + z_m,
\label{fgka}\end{equation}
where
\[
 c_m = \int h_m(y) dy, \qquad  d_m^i = \int y_i h_m(y) dy
\]
\[
 \qquad z_m=  \chi_{> m+2}  \int h_m(y)  \left(\frac{1}{|x-y|} - \frac{1}{|x|} - \frac{y \cdot x}{|x|^2}\right) dy.
\]
Then we estimate
\[
2^{m} |c_m| + |d_m^i| + \|z_m\|_{Z^{N,0}} \lesssim 
\|\la r \ra^2 h_m\|_{\LE^*}
\]
Hence we can set 
\[
c = \sum_m c_m, \qquad d^i(r) = \sum_{m} d_m^i \chi_{> m+2}(r),
\qquad e(r) =  r \sum_m c_m \chi_{<m+2}(r)  
\]
One easily checks that $c$, $d^i$ and $e$ satisfy bounds as in 
\eqref{fdr}, so the proof of \eqref{fdr} is concluded in this case.

It remains to consider the case when $h \in S_{rad}(r^{-4})$.
Then $w$ is also smooth and spherically symmetric, and we can write
\[
 \partial_r^2 (rw) = rh
\]
Integrating this relation twice, first from infinity and the 
second time from zero, it follows that
$w$ has the form
\[
w_2 = c \la r \ra^{-1} + e(r) \la r\ra^{-2}, \qquad   e \in S_{rad}(1)
\]
Furthermore, since $h$ is supported in $r \gtrsim R$, we obtain
\[
 |c| \lesssim R^{-1}
\]
Thus the proof of \eqref{fdr} for $w$ is concluded.
\end{proof}

The next step is to use the $R_0$ estimates in order to improve the 
analysis of $R_\tau g$ for $\tau$ near $0$, and in particular the 
bounds in Proposition~\ref{putau}, in the case when $g$ has better
decay at infinity. Precisely, as in the previous proposition,
we will successively consider the case when $g \in Z^{k+4,1}$ and 
when $g \in Z^{k+4,2}$. We will use the notation $a \wedge b$
for a smooth $1$-homogeneous function of $a, b \in \R^+$
which equals $\min\{a,b\}$ unless $a \approx b$.

\begin{proposition}
a) Let $g \in Z^{k+4,1}$  and $|\tau| \lesssim 1$ with 
$\Im \tau \leq 0$. Then the function $v(\tau) = R_\tau g$ can be represented as 
\[
 v(\tau) = e^{-i\tau \la r\ra } v(0) + w_1(\tau)
\]
where the function $w_1$ is given by
\[
 w_1(\tau) = R_\tau (\chi_{> |\tau|^{-1}}(r) g + \tau h)
\]
with $h$ satisfying
\begin{equation}
 \|r^m (\partial_r +i\tau)^m T^i \Omega^j S^l h \|_{\LE^{*}} \lesssim 1,
\qquad m+i+ j+ l \leq k
\label{hbd-tau}\end{equation}

b) Assume in addition that $g \in Z^{k+4,2}$. Then $v$ can be represented as
\[
  v(\tau) = (v(0)+\tau v_1 + \tau e_0(r,\tau)) e^{-i\tau r}  + w_2(\tau)
\]
where 
\[
 v_1 = R_0 g_1, \qquad g_1 \in Z^{k,1},
\]
the radial function $e_0$ has the form
\[
 e_0(r,\tau) = r^{-1} e_1(r \wedge |\tau|^{-1}) + \tau 
( e_2(r \wedge |\tau|^{-1}) - e_2(|\tau|^{-1})), \qquad e_1, e_2 \in S(\log r)
\]
and $w_2$ is given by
\[
 w_2(\tau) = R_\tau (\chi_{> |\tau|^{-1}} g + \tau^2 h)
\]
with $h$ satisfying \eqref{hbd-tau}.

\label{pusmall}
\end{proposition}

\begin{proof}
a) We note that the function $e^{-i\tau \la r\ra } v(0)$ belongs to the
range of $R_\tau$ since it satisfies the outgoing radiation condition
\eqref{outrad}. We compute the equation for the
 function $w_1= v(\tau) - e^{-i\tau r} v(0)$:
\begin{equation}
\begin{split}
P_\tau w_1 = &\ g+  (g e^{-i\tau \la r \ra}-g) + 2\tau (\partial_r +\frac{1}r)v(0) e^{-i\tau \la r \ra}\\ &\ + \left(\tau l^1S(r^{-2})+ \tau^2 l^1 S (r^{-1})+ \tau l^1 S(r^{-1}) \nabla\right) v(0)
 e^{-i\tau \la r \ra} 
\end{split}
\label{ptauw}\end{equation}
where the second term on the right comes from the Laplacian and the 
third is generated by the short range part of $P$. We observe that the 
long range part of $P$ does not produce any contributions because
the function $e^{i\tau \la r\ra}$ is radial. 

For $v(0)$ we use the representation \eqref{r0d} and the corresponding bounds \eqref{r0da}. This allows us to estimate all terms on the right but 
the first one as in \eqref{hbd-tau}. We note in particular that
in order to estimate the contribution of the first term in \eqref{r0d}
to the second term above, it is necessary to take advantage of the improved $l^1S(1)$ bound on $S_r c$ in \eqref{r0da}.

Finally we consider the first term on the right in \eqref{ptauw},
which we rewrite as
\[
 g e^{-i\tau \la r\ra}-g = \chi_{< |\tau|^{-1}} g (e^{-i\tau \la r\ra}-1)
+ \chi_{> |\tau|^{-1}}g e^{-i\tau \la r\ra} - \chi_{> |\tau|^{-1}}g
\]
In the first term we can pull out a $\tau \la r\ra$ factor from the difference
$1-e^{-i\tau \la r\ra}$, and then \eqref{hbd-tau} follows. For the 
second factor,  \eqref{hbd-tau} is obtained directly.

b) We start again with \eqref{ptauw}, but now we can only place into $h$
the terms which have a $\tau^2$ factor; we call such terms {\em negligible 
errors}. This includes the middle term
in the last expression in \eqref{ptauw}. Now $v$ is as in \eqref{r0e}-\eqref{r0f},
therefore  the remaining part of the same expression can be written in the form
\[
 \tau ( l^1 S(r^{-2})+ l^1S(r^{-1}) \nabla ) v(0)e^{-i\tau \la r\ra}:= \tau g_1 e^{-i\tau \la r\ra}, \qquad g_1 \in Z^{k,1}
\]
and use the argument in part (a). Since $g_1$ already has a factor 
of $\tau$, the error in part (a) will have a factor of $\tau^2$.
We further note that, compared to part (a), the error 
$\tau g_1 e^{-i\tau \la r\ra}$ is better because it already contains 
the oscillatory factor $e^{-i\tau \la r\ra}$; we call such terms 
{\em type $1$ negligible errors}.

It remains to consider the second term in \eqref{ptauw}. This is where
we need to take advantage of the more precise representation 
for $v(0)$ in \eqref{r0e}-\eqref{r0f}. The contribution 
of $c r^{-1}$ vanishes, while the contribution of $v_2$ in \eqref{r0e}
is type $1$ negligible.

Consider now the error term $\tau 
(\partial_r +\frac{1}{r}) \left( d_i(r) \partial_i r^{-1}\right)e^{-i \tau \la r \ra}$,
arising from the second term in \eqref{r0e}. We compensate
for it with the explicit correction
\[
 \tau v_2 = \tau d_i(r) \partial_i (\ln r) e^{-i\tau \la r\ra}
\]
A direct computation shows 
that 
\[
\begin{split}
 P_\tau v_2 - (\partial_r +\frac{1}{r}) \left( d_i(r) \partial_i r^{-1}\right) e^{-i \tau \la r\ra} = &\ 
(P_\tau - (\Delta + \tau^2)) v_2  \\ &\ +
\partial_r d_i (\partial_i r^{-1}) e^{-i \tau \la r \ra}
+ \partial_r^2 d_i \partial_i (\ln r) e^{-i\tau \la r\ra } 
\end{split}
\]
It suffices to show that the right hand side can be expressed in the form
$g_2  e^{-i\tau \la r\ra } + \tau h $ with $g_2 \in l^1S(r^{-3})$ and $h$ as in \eqref{hbd-tau}. This follows by taking into account the 
effect of the short and long range terms 
in $P_\tau$ for the first term, and by using the bound on $S_r{d_i}$ in \eqref{r0f} for the last two terms.

Finally we consider the term $\tau (\partial_r +\frac{1}r) e(r) e^{-i\la r\ra \tau} \in S^{-3} e^{-i\la r\ra \tau}$. This is the term which arises due to the long range part of $P_\tau$, and which is in 
effect responsible for the decay rate in our main theorem.
Here there  is no mechanism that can eliminate  logarithmic type 
corrections for the resolvent. The redeeming feature of this case 
is that we can work fully with spherical symmetry, and that we 
can replace $P_\tau$ by $\Delta+\tau^2$ modulo negligible perturbations.

We begin by solving the equation which corresponds to $\tau = 0$,
namely
\[
\Delta v_3 = (\partial_r +\frac{1}r) e(r) \in S(r^{-3})
\]
and we obtain a solution of the form $v_3 = \la r\ra ^{-1} e_1(r)$ with $e_1 \in S_{rad}( \log r)$.
However, unlike all previous cases, here the function 
$\tau v_3 e^{-i\tau \la r \ra}$ can no longer be a suitable correction, because that would inherit the $\log r$ term even when
$\tau > r^{-1}$. Instead we consider the modified function $\tilde v$ 
\[
\tilde v(r,\tau) = \la r \ra^{-1} e_1(r \wedge |\tau|^{-1})
\]
and use $v_3=\tau \tilde v(\tau,r) e^{- i\tau \la r\ra}$ as the correction.
A direct computation shows that
\[
P_\tau v_3 = \tau (\partial_r +\frac{1}r) e(r) e^{-i \tau \la r\ra} + 
2\tau^2 (\partial_r +\frac{1}r) \tilde v e^{-i \tau \la r\ra}
+ \tau g_3(r) e^{-i \tau \la r\ra} + \tau^2 h
\]
Here $g_3 \in l^1 S(r^{-3})$ and $h$ as in \eqref{hbd-tau}
are negligible errors generated by the long range and short range 
part of $P_\tau$.
Thus only  the second term provides non-negligible errors and 
needs further attention.
Modulo terms which can be included in $h$,
the expression $(\partial_r +\frac{1}r) \tilde v$ has the form 
\[
(\partial_r +\frac{1}r) \tilde v = \chi_{< |\tau|^{-1}}(r) 
\la r \ra \partial_r e_1(r) +  h
\]

Even though this comes with a $\tau^2$ factor, it is not quite 
negligible because it only has decay $\la r \ra^{-1} \partial_r e_1 \in S(r^{-2})$  at infinity so it logarithmically fails to be in $\LE^*$
uniformly with respect to $\tau$. Instead we need to reiterate again. 

We solve 
\[
 \Delta e_2 = \la r \ra^{-1} \partial_r e_1 
\]
and obtain a solution $e_2 \in S(\log r)$. Then we choose the last
correction of the form
\[
 v_4 = \tau^2 \la r \ra^{-1} ( e_2(r \wedge |\tau|^{-1}) - e_2(|\tau|^{-1}))
 e^{-i\tau \la r\ra}
\]
Here it makes no difference whether we add the factor $e^{-i\tau r}$ 
or not, as $v_4$ is supported in $\{ r |\tau| \lesssim 1\}$ and does not 
see any oscillations. Again, a direct computation shows that
\[
 P_\tau v_4 = \tau^2 \chi_{< |\tau|^{-1}}(r) 
\la r \ra \partial_r^{-1} e_1 e^{-i\tau \la r \ra} + \tau^2 h 
\]
with $h$ as in \eqref{hbd-tau}, a negligible error. 
Now the first term has a better $\tau^2$ factor, so it is also negligible. 
The proof of the proposition is concluded.

\end{proof}

\section{ From $L^2$ to pointwise bounds}
\label{linf}
In this section we supplement the weighted $L^2$ resolvent bounds with
corresponding pointwise bounds. Given a $\tau$ dependent function
$g(\tau)$ which is defined in $\Im \tau \leq 0$ or in an open subset
of it, we look for pointwise bounds for $R_\tau g$ and its
derivatives.  First we consider the case when $g$ has only $\LE^*$
type decay at spatial infinity.

\begin{proposition}

Let $g \in \LE^*$, depending on $\tau$, so that \eqref{oslarge} holds. 

(i) Large $\tau$, $|\tau| \geq 1$.  Then the following  pointwise bounds are valid:
\begin{equation}
| T^i \Omega^j S^k R_\tau g|\lesssim \la r\ra^{-1}, 
\qquad i+4j+16k \leq N -12
\label{pointhigh}\end{equation}

(ii) Small $\tau$, $|\tau| \leq 1$.  Then the following  pointwise bounds are valid,
\begin{equation}
| T^i \Omega^j S^k R_\tau g|\lesssim \left\{
\begin{array}{lc} \min\{1,(|\tau| \la r\ra )^{-1}\} & i = 0 \cr
\la r\ra^{-1} & i=1 \cr \la r\ra^{-2} + |\tau| \la r\ra^{-1} & i \geq 2
\end{array} \right.\qquad i+4j+16 k \leq N -12
\label{pointlow}\end{equation}

(iii) In addition, if $\tau$ is real then
the outgoing radiation condition holds:
\begin{equation}
\lim_{|x| \to \infty} r (\partial_r + i \tau) 
T^i \Omega^j S^k R_\tau g = 0 \qquad i+4j+16k \leq N -12
\label{radpoint}\end{equation}
\label{p:point}
\end{proposition}

\begin{proof}
(i) Denote $v = R_\tau g$ and 
$g_{ijk} = T^i \Omega^j S^k R_\tau g$,
$v_{ijk} = T^i \Omega^j S^k R_\tau g$.
By Proposition~\ref{putau},  $g_{ijk}$ satisfy \eqref{ossmall}.
 Using the Sobolev embeddings on the sphere
\[
 \|\phi\|_{L^\infty(\S^2)} \lesssim \|\phi\|_{L^2(\S^2)}+ \| \Omega^2 \phi\|_{L^2(\S^2)}
\]
we obtain for $i+4j+16k < N-12$:
\begin{equation}
 \sum_m  2^{\frac{3m}2} \| g_{ijk} \|_{L^2_r L^\infty_\omega(A_m)} 
\lesssim 1, \qquad  i+4j+16k < N-8
\label{pg}\end{equation}
respectively
\begin{equation}
 \| v_{ijk} \|_{L^2_r L^\infty_\omega(A_m)}
\lesssim  2^{-\frac{m}2} , \qquad i+4j+16k < N-12
\label{pu}\end{equation}
Using these bounds we revisit the equation for $v_{ijk}$
which we can rewrite in the form
\begin{equation}
 (\partial_r^2 + \tau^2)(r v_{ijk}) = r^{-1} \Delta_\omega v_{ijk}
+ r Q_{lr} v_{\leq i,\leq j,\leq k} + r Q_{sr} v_{\leq i,\leq j,\leq k}
+ r g_{\leq i,\leq j,\leq k}
\label{longeq}\end{equation}
In view of \eqref{pg} and \eqref{pu}  we obtain
\begin{equation}
 \sum_m 2^{\frac{m}2} \| (\partial_r^2 + \tau^2)(r v_{ijk}) \|_{L^2_r L^\infty_\omega(A_m)}
\lesssim 1, \qquad i+4j+16k < N-20
\label{leas}\end{equation}
Using \eqref{pu} for $v_{ijk}$ and $\partial_r v_{ijk}$ we can localize
$v_{ijk}$ so that the above bound is preserved,
\[
 \| (\partial_r^2 + \tau^2)(\chi_{A_m} r v_{ijk}) \|_{L^2_r L^\infty_\omega(A_m)}
\lesssim 2^{-\frac{m}2} 
\]
Now we use the following fundamental solution for $\partial_r^2 - \tau^2$,
\[
 K_\tau(s) = \tau^{-1} e^{i\tau|s|}
\]
This satisfies the bounds
\[
 |K_\tau(s)| \lesssim |\tau|^{-1}, \qquad |\partial_s K_\tau(s)| \lesssim 1
\]
Then the last bound  leads to 
\[
|\tau| \| \chi_{A_m} r v_{ijk}\|_{L^\infty(A_j)} +  \|\partial_r (\chi_{A_m} r v_{ijk})\|_{L^\infty(A_j)}
\lesssim 1
\]
which we rewrite in the form
\begin{equation}
|\tau| |v_{ijk}| + |\partial_r v_{ijk}| \lesssim \la r\ra^{-1}
\end{equation}
This concludes the proof of \eqref{pointhigh}.

(ii) The arguments above remain valid in dyadic regions $A_m$ with
$2^m > |\tau|^{-1}$, and directly yield \eqref{pointlow} for
$i = 0$ and $i=1$; in the latter case we need the additional observation
that $T$ can be expressed as a bounded linear combination
of $\partial_r$ and $r^{-1} \Omega$. 

 Similarly, we can express $T^2$ in terms of $\partial_r^2$, 
$r^{-1} \partial_r \Omega$ and $r^{-2} \Omega^2$. Thus, in view of
\eqref{pointlow} for $i=0,1$, for the $i = 2$ case of \eqref{pointlow} it suffices to bound  $\partial_r^2 \Omega^j S^k v$. For this we 
return to \eqref{longeq}, move the term $\tau^2 v_{ijk}$ on the right
and use \eqref{pointlow} for $i=0,1$. This allows us to 
obtain a favourable bound for all terms on the right except for
the $\partial_r^2$ part of $Q_{sr}$. We obtain
\[
\| \partial_r^2 v_{0jk} \|_{L^\infty(A_m)} \lesssim 
|\tau| \la r\ra^{-1} + 2^{-m} \| \partial_r^2 v_{0,\leq j,\leq k} \|_{L^\infty(A_m)}
\]
The second term on the right does not appear when $j=0$ and $k=0$.
Hence we can use induction with respect to  $j,k$ to obtain
\[
\| \partial_r^2 v_{0jk} \|_{L^\infty(A_m)} \lesssim 
|\tau| \la r\ra^{-1}
\]
and conclude the proof of \eqref{pointlow} when $i=2$.
The case $i > 2$ is  identical.

It remains to consider the case when $m$ is small, namely
$2^{m} \lesssim |\tau|^{-1}$. There the argument is slightly different.
Again by the Sobolev embeddings on the sphere, from \eqref{ossmall} we obtain
\begin{equation}
2^{-m} \| v_{ijk} \|_{L^2_r L^\infty_\omega(A_m)} +
 \|\nabla v_{ijk} \|_{L^2_r L^\infty_\omega(A_m)}+
2^{m} \|\nabla^2 v_{ijk}\|_{L^2_r L^\infty_\omega(A_m)}
\lesssim  2^{-\frac{m}2} 
\label{pu2}\end{equation}
Then the bound \eqref{pointlow} for $i=0,1$ 
follows by Sobolev embeddings  with respect to the radial variable.
Finally we obtain \eqref{pointlow} for $i \geq 2$ directly
from the equation \eqref{longeq} as above.

iii) Suppose $\tau$ is real. From the spherical Sobolev embeddings
and the $L^2$ radiation condition \eqref{outrad} applied
to the functions $v_{ijk}$ we obtain
\begin{equation}
 \lim_{m \to \infty} 2^{\frac{m}2}
\| (\partial_r +i \tau) v_{ijk}\|_{L^2_r L^\infty_{\omega}} 
= 0
\label{aveout}\end{equation}
Writing $\partial_r^2+ \tau^2 = (\partial_r - i \tau)(\partial_r
+i\tau)$ we interpret \eqref{leas} as a first order ode for the
functions $(\partial_r +i \tau) (r v_{ijk})$. The averaged bound
\eqref{aveout} allows us to integrate this ode from infinity to obtain
\[
 \lim_{r \to \infty} (\partial_r +i \tau) (r v_{ijk}) = 0
\]
which implies \eqref{radpoint}.

\end{proof}

Next we consider a more favourable case when $g$ either has better
decay at spatial infinity, or it has an oscillatory behavior which
matches the one of the resolvent. For simplicity we restrict ourselves
to real $\tau$. If $\tau$ is large then only the former alternative
needs to be considered. In this case we will prove the following:

\begin{proposition}
  Let $g \in Z^{m,n}$ with $m \gg n$. Then for real $\tau$ with
  $|\tau| \gtrsim 1$, the function $v = R_\tau g$ satisfies the
  following pointwise bounds:
\begin{equation}
 |(\tau \partial_\tau)^l  (v e^{i \tau \la r\ra})| \lesssim  
|\tau|^{l} \la r\ra^{-l-1}, \qquad 
l \leq n
\label{pointlarge}\end{equation}
\label{p:pointlarge}\end{proposition}

\begin{proof}
We first observe that by Proposition~\ref{p:point}, the bounds
\eqref{pointhigh} hold with $N=m$. 
  Also, by Sobolev embeddings, $g$ satisfies the stronger
pointwise bounds
\begin{equation}
|T^i g| \lesssim \la r \ra^{-2-i}, \qquad i \leq n-2
\label{gzmn}\end{equation}

For $x$ in a compact set the exponential $e^{i \tau \la r\ra}$ 
is harmless, and on the other hand
we can write $\tau \partial_\tau$ as a smooth combination 
of $S$ and $\tau T$. Hence  the bound \eqref{pointlarge}  follows 
from \eqref{pointhigh}. It remains to prove \eqref{pointlarge} 
for large $r$.

Applying \eqref{pointhigh} for $i=0$, $j=0$, $k=0$ we directly obtain
\eqref{pointlarge} for $l=0$. For $l=1$ we compute
\[
 \tau \partial_\tau  (v e^{i \tau r})= S  v e^{i \tau r} +
 r (\partial_r+i\tau) v  e^{i \tau r} 
\]
The first term is estimated again by \eqref{pointhigh}.
Hence we need to prove that for large $r$ we have
\begin{equation}
 |(\partial_r+i\tau)  v | \lesssim \frac{|\tau|}{r^2}
\label{dv}\end{equation}
For $ r u$ we have in polar coordinates
the equation
\[
 (\partial_r^2 + \tau^2) (r  v) = r g - r^{-1} \Delta_{\omega} v 
+ r Q_{lr}  v  + r Q_{sr} v
\]
Using \eqref{pointhigh} and \eqref{gzmn} we can bound the RHS by $|\tau| r^{-2}$. Then
\[
 | (\partial_r-i\tau)(\partial_r+i\tau) (r v)| \lesssim |\tau| r^{-2} 
\]
which is integrated from infinity using the radiation condition to
obtain \eqref{dv}.  The same argument applied to the equation
\eqref{longeq} for $v_{ijk}$ instead of $v$ yields
\begin{equation}
 |(\partial_r+i\tau)  v_{ijk} | \lesssim \frac{|\tau|}{r^2}, \qquad i+4j+16k \leq m-20
\label{dvijk}\end{equation}

For $l=2$ we compute 
\[
(\tau \partial_\tau)^2 (v e^{i \tau r}) = S^2  v e^{i \tau r} + 
r (\partial_r+i\tau)  S v e^{i \tau r}+r^2  (\partial_r+i\tau)^2  
 v e^{i \tau r}
\]
The second term is handled via \eqref{dvijk}
 so it remains to prove that
\begin{equation}
 |(\partial_r+i\tau)^2  v | \lesssim \frac{\tau^2}{r^3}
\label{dv2}\end{equation}
We have
\[
\begin{split}
(\partial_r-i\tau)(\partial_r+i\tau)^2 (r v)
= &\ (\partial_r+i\tau)(r g + r^{-1} \Delta_{\omega}  v 
+ r^{-2}  v  + r Q_{sr}  v)
\\
= &\ ( g + r^{-2} \Delta_{\omega}  v 
+ r^{-3}  v  +  Q_{sr} u) + r (\partial_r+i\tau) f
\\ &\
 + r^{-1} \Delta_{\omega} (\partial_r+i\tau)  v 
+ r^{-2} (\partial_r+i\tau)  v  + r Q_{sr} (\partial_r+i\tau) v
\end{split}
\]
Since $(\partial_r+ i\tau) v $ also satisfies the radiation condition
\eqref{radpoint}, it suffices to show that the RHS admits a $\tau^2 r^{-3}$
bound. This follows from the pointwise bounds \eqref{pointhigh} for
$u$ for the terms on the first line, and \eqref{dvijk} for the terms
on the second line.

A similar computation using \eqref{longeq} also shows that for large $r$ we have 
\begin{equation}
 |(\partial_r+i\tau)^2  v_{ijk} | \lesssim \frac{\tau^2}{r^3}
, \qquad i+4j+16k \leq m-40
\label{dv2ijk}\end{equation}
The proof of \eqref{pointlarge} is completed by induction with respect to $l$.
\end{proof}

Next we consider the case of small $\tau$. Here instead 
of the stronger decay assumption at spatial infinity we consider 
a weaker condition for $g$ and establish two types of regularity
bounds with respect to the spectral parameter $\tau$:

\begin{proposition}
  Let $m$ be a large integer. Let $g$ be a $Z^{m,0}$ valued function
  of $\tau$ in $ [-1,1] $ which satisfies the additional bounds:
\begin{equation}
\|\la r \ra^n (\partial_r+ i\tau)^n T^i \Omega^j S^k g \|_{\LE^*}  \lesssim 1, \qquad n+i+j+k \leq m
\label{glowreg}\end{equation}
Then the  function $v = R_\tau g$ has the following properties:

(i) $r \lesssim  |\tau|^{-1}$. Then we have
\begin{equation}
 |(\tau \partial_\tau)^l  v| \lesssim 1, \qquad  l \ll m
\label{smallr}\end{equation}

(ii) $r \gtrsim |\tau|^{-1}$. Then we have
\begin{equation}
 |(\tau \partial_\tau)^l  (v e^{i\tau \la r \ra})| \lesssim (|\tau| \la r\ra)^{-1}, 
\qquad l \ll m
\label{larger}\end{equation}

\label{p:larger}\end{proposition}

\begin{proof}
To prove the proposition we rely on the estimate \eqref{pointlow} applied to the 
function $v$. For $g$ on the other hand, by Sobolev embeddings
we have the pointwise bound
\begin{equation}
 | T^i \Omega^j S^k  (g e^{i\tau \la r \ra}) | 
\lesssim \la r \ra^{-2-i}, \qquad i+j+k \leq m-2
\label{pointg}\end{equation}

(i). For $l=0$ the bound \eqref{smallr} follows directly from 
\eqref{pointlow}. Consider now the case $l=1$. For small $r$ we write
\[
 S_\tau v =  - S v +  S_r v
\]
which is again bounded by \eqref{pointlow}. Even better, \eqref{pointlow}
shows that for the functions $v_{ijk}$ we have 
\begin{equation}
| S_\tau v_{ijk}| + |S_r v_{ijk}| \lesssim 1, \qquad i+4j+16k \leq m-20. 
\label{f1}\end{equation}

For $l=2$ we write 
\[
 S_\tau^2 v_{ijk} = S^2 v_{ijk} - 2  S_r S v_{ijk}
+ S_r^2 v
\]
 The first two terms are bounded as before by \eqref{pointlow}
and \eqref{f1}. For the third we use 
the equation \eqref{longeq} for $v_{ijk}$, which gives
\begin{equation}
 |S_r^2 v_{ijk}| 
\lesssim r^2 \tau^2 |v_{ijk}| + |S_r v_{ijk}|  +  |\Omega^2 v_{ijk}| 
+  r^{-1} \sum_{a+b \leq 2} |S_r^a \Omega^b v_{\leq i,\leq j,\leq k}| 
+ r^2 |g_{\leq i,\leq j,\leq k}|
\label{rdr2}\end{equation}
We bound all the terms on the right via \eqref{pointlow} and 
\eqref{f1} to obtain
\begin{equation}
 |S_r^2 v_{ijk}| \lesssim 1, \qquad i+4j+16k \leq m-40. 
\label{f2}\end{equation}

Now we repeat the $l=2$ argument to prove by induction that
\begin{equation}
 |S_r^l v_{ijk}| \lesssim 1, \qquad i+4j+16k \leq m-20 l. 
\label{f2bis}\end{equation}
We remark that at  each application of the equation $P_\tau v = g$, 
one or two  $S_r$ factors  are replaced by either an $\Omega$ factor 
or by $r \tau$ or simply by $1$.

(ii) Again the case $m=0$ follows directly from \eqref{pointlow}.
Consider now the case $l\geq 1$. For large $r$ we compute
\[
 \tau \partial_\tau (v e^{i\tau r}) = Sv\ e^{i\tau r} + r ( \partial_r -
i \tau) v e^{i \tau r}
\]
Hence if $l=1$ then we need to show that 
\begin{equation}
 |(\partial_r +i \tau) v| \lesssim |\tau|^{-1} r^{-2}, \qquad r > |\tau|^{-1}
\label{lowk=1}\end{equation}
Repeating the computation at the second and higher levels, the 
conclusion follows if we prove the more general bound
\begin{equation}
  |(\partial_r +i \tau)^l v_{ijk}| \lesssim |\tau|^{-1} r^{-1-l},  \qquad r > |\tau|^{-1},
\quad i+4j + 16 k + 20 l \leq m
\label{lowk>1}\end{equation}

For $l=1$ we rewrite the equation for $v$ in the form
\begin{equation}
(\partial_r -i \tau)(\partial_r +i \tau) (r v) =  r^{-1} \Delta_{\omega} v
+ r Q_{lr} v  + r Q_{sr} v + rg
\label{polar}\end{equation}
By the outgoing radiation condition \eqref{radpoint} we can view this 
as an ode for $(\partial_r +i \tau) (r v)$ and integrate  from infinity.
By \eqref{pointlow}  we can bound the first 
three RHS terms by $|\tau|^{-1} r^{-2} $ and integrate
to obtain $(|\tau| r)^{-1}$ as desired. For the $g$ term a further 
integration by parts is needed. Precisely, we have
\[
 \int_{r_0}^\infty r g(r,\omega) e^{-i \tau r} dr =
- \frac{i}2 \tau^{-1} r_0 g(r_0,\omega)e^{-i \tau r_0} +
\int_{r_0}^\infty \frac{i}2 \tau^{-1} r e^{-2i \tau r} \partial_r (g(r,\omega)e^{ir\tau})  dr 
\]
Hence by \eqref{pointg} we obtain the $(|\tau| r)^{-1}$ bound,
and\eqref{lowk=1} is proved. 
We further observe that we can apply the same argument 
to $v_{ijk}$ using the equation \eqref{longeq}
to obtain
\begin{equation}
 |(\partial_r +i \tau)  v_{ijk}| \lesssim r^{-2}, \qquad r > |\tau|^{-1}
\label{lowk=1a}\end{equation}

For $l=2$ we apply another $\partial_r -i\tau$ operator in \eqref{polar}.
This yields
\[
\begin{split}
 (\partial_r^2 + \tau^2) (\partial_r + i \tau)(r v) = &\
  r^{-2} \Delta_{\omega} (\partial_r +i \tau)v
+ r Q_{lr} (\partial_r +i \tau)v  + r Q_{sr} (\partial_r +i \tau) v
\\
&\
+ Q_{lr} v  +  Q_{sr} v + (\partial_r +i \tau)(r g)
\end{split}
\]
In view of \eqref{pointlow} and \eqref{lowk=1a} we can bound all terms
on the right by $|\tau|^{-1} r^{-3}$ except for the last one, and obtain
\eqref{lowk>1} for $l=2$ integrating from infinity. For the last term,
an additional integration by parts is needed, exactly as
above. Applying the same argument to $v_{ijk}$ we also obtain
\eqref{lowk>1} for $l=2$.  The bound for higher $l$ is proved in an
inductive manner.
\end{proof}

\section{Conclusion }

By Lemma~\ref{l:cp}, the time Fourier transform $\hat u$ of the solution $u$ to the Cauchy problem \eqref{ahom} is represented as
\begin{equation}
 \hat u(\tau) = R_\tau (\tau u_0 + 2P_1 u_0 + u_1):= R_\tau(\tau f + g),
\qquad f \in  Z^{m+1,1}, \qquad g \in  Z^{m,2}.
\label{hatu}\end{equation}
The bound \eqref{pd} will be obtained from pointwise bounds 
on $\hat u(\tau) $ and its derivatives with respect to $\tau$ by reverting the Fourier transform. This requires slightly different arguments for small
and for large $\tau$. 

\subsection{The case of large $\tau$, $|\tau| \gtrsim 1$.} If $\tau$ is large, then a good low frequency approximation for $\hat{u}$ is 
$ \hat u \approx \tau^{-1} f$. Subtracting that,
we have 
\[
\hat  u (\tau) = \tau^{-1} f + u_1
\]
where
\[
 P_\tau u_1 =  (P^1 f + g) + \tau^{-1} P^2 f:= f_1+\tau^{-1}g_1
\]
with $f_1 \in H^{m,2}$ and $f_2 \in H^{m-1,3}$.
Reiterating we obtain the representation
\begin{equation}
 \hat u = \sum_{k=0}^{K-1} \tau^{-k-1} f_k + \tau^{-K} R_\tau (\tau f_K +g_K) 
\label{uexp}\end{equation}
where $f_0 = f$, $f_k \in H^{m+1-k,k+1}$, 
and $g_k \in H^{m+1-k,k+2}$. Here we choose $K$ large but with $K \ll m$.
 This leads  to a decomposition 
of the high time frequencies $P_{> 1}(|D_t|) u$ in $u$ of the form
\[
P_{> 1}(|D_t|) u = u_a + u_b, \quad \hat u_a = \chi_{>1}(|\tau|)
\sum_{k=0}^{K-1} \tau^{-k-1} f_k, \quad \hat u_b = \chi_{>1}(|\tau|)
\tau^{-K} R_\tau (\tau f_K +g_K). 
\]
For the summands we have $|f_k| \lesssim \la r\ra ^{-3-k}$ therefore 
applying an inverse Fourier transform  we obtain
\begin{equation}
 |u_a| \lesssim \la t \ra^{-N} \la r\ra^{-3}
\label{uabd}\end{equation}
 For the resolvent expression in $u_b$ we use Proposition~\ref{p:pointlarge}. We obtain
\[
 | (\tau \partial_\tau)^l (\hat u_b(\tau)e^{i\tau r})| \lesssim  |\tau|^{-K+l} 
\la r \ra^{-1-l}.
\]
Then, inverting again the Fourier transform,
we obtain the following bound:
\begin{equation}
|u_b| \lesssim  \la r\ra^{-1} \la t-r \ra^{1-K}
\label{ubbd}\end{equation}
Together, \eqref{uabd} and \eqref{ubbd} give the high frequency part
of \eqref{pd}.

To bound the high frequency part of $\partial_t u$ the same argument
applies with one exception, namely the first term $f \chi_{> 1}(\tau)$
arising in the expression for $\widehat{\partial_t u_a}$.  However,
such a term is natural since $u$ is extended by $0$ to negative times
therefore $\partial_t u$ contains a $ f \delta_{t=0}$ component. Factoring
that out we obtain as desired
\begin{equation}
|P_{> 1}(|D_t|) (\partial_t u_a - f \delta_{t=0})| \lesssim \la t \ra^{-N} \la r\ra^{-4},
\label{uabdt}\end{equation}
i.e. the high frequency part of \eqref{pdt}.

\subsection{The case of small $\tau$, $|\tau| \ll 1$.}
Here we will still need the arguments from the large $\tau$ case, but there
is an additional layer of the proof, given by the resolvent expansion 
around $\tau=0$ provided by Proposition~\ref{pusmall}. Since $f \in Z^{m+1,1}$ and $g \in Z^{m,2}$, by applying Proposition~\ref{pusmall} (a)
for $f$ and Proposition~\ref{pusmall} (b) for $g$ we obtain the 
representation
\[
\hat u(\tau) = (v + \tau v_1 + \tau e(r,\tau)) e^{i\tau r} + 
R_\tau ( \chi_{> |\tau|^{-1}} (\tau f+g)) + \tau^2 R_\tau h
\]
with $h$ as in \eqref{hbd-tau}. Correspondingly we decompose
\[
 P_{< 1}(|D_t|) u =  u_a + u_b + u_c + u_d 
\]
where
\[
 \hat u_a(\tau) = \chi_{< 1}(|\tau|) R_\tau ( \chi_{> |\tau|^{-1}} (\tau f+g)),
\qquad \hat u_b (\tau)= \tau^2 \chi_{< 1}(|\tau|) R_\tau h
\]
\[
 \hat u_c (\tau)= \chi_{< 1}(|\tau|) (v + \tau v_1)e^{i\tau r},
\qquad 
\hat u_d (\tau)= \tau \chi_{< 1}(|\tau|) e(r,\tau)) e^{i\tau r}
\]

For the first term $u_a$ we use an expansion as in 
\eqref{uexp} to write
\[
 \hat u_a  = \chi_{< 1}(|\tau|) \sum_{k=0}^{K-1} \tau^{-k-1} \chi_{> |\tau|^{-1}} f_k + \tau^{-K} \chi_{< 1}(|\tau|) R_\tau (\chi_{> |\tau|^{-1}}(\tau f_K +g_K)) 
\]
Since $f_k \in H^{m+1-k,k+1}$ and $g_k \in H^{m+1-k,k+2}$, the second
term above can be included in $\hat u_b$. We remark that the 
spatial cutoff $\chi_{> |\tau|^{-1}}$ plays an essential role,
as it allows the spatial decay of $f_K$ and $g_K$ to compensate
for the $\tau^{-K}$ factor. Hence we redenote
\[
 \hat u_a  = \chi_{< 1}(|\tau|) \sum_{k=0}^{K-1} \tau^{-k-1} \chi_{> |\tau|^{-1}} f_k
\]
Using the pointwise bounds $|f_k| \lesssim \la r \ra^{-3-k}$ we 
invert the Fourier transform to obtain as in the large $\tau$ case 
\[
 |u_a| \lesssim \la r \ra^{-3} \left( \frac{\la r \ra}{\la r\ra +t}
\right)^{N}, \qquad  |\partial_t u_a - P_{< 1}(|D_t|) f\delta_{t=0}| \lesssim \la r \ra^{-3} \left( \frac{\la r \ra}{\la r\ra +t}
\right)^{N}
\]

For $u_b$ we use Proposition~\ref{p:larger} to obtain
\[
 |(\tau \partial_\tau)^l \hat u_b| \lesssim \tau^2, \qquad |\tau| < \la r \ra^{-1}
\]
and
\[
 |(\tau \partial_\tau)^l (\hat u_b e^{ir\tau})| \lesssim \tau r^{-1}, \qquad |\tau| > \la r \ra^{-1}
\]
which lead to
\[
 | u_b| \lesssim \la r \ra^{-1} \la t- r \ra^{-2}, \qquad 
| \partial_t u_b| \lesssim \la r \ra^{-1} \la t- r \ra^{-3}
\]

For $u_c$ we have the pointwise bounds $|v_0| \lesssim 1$, $|v_1| \lesssim 1$ therefore we easily get
\[
 |u_c| + |\partial_t u_c| \lesssim \la r \ra^{-1} \la t- r \ra^{-N}
\]
 
Finally we consider $u_d$, which in this analysis represents the 
leading term of the contribution to $u$ of the long range terms
in the operator $P$.

We recall that the expression $e(r,\tau)$ has the form
\[
 e(r,\tau) = \la r\ra^{-1} e_1(r \wedge |\tau|^{-1}) + \tau 
(e_2(r \wedge |\tau|^{-1}) - e_2(|\tau|^{-1})), \qquad e_1, e_2 \in S(\log r)
\]
We represent $e_1$ and $e_2$ in the form
\[
 e_{12}(r) = \sum_{k \geq 0} c_k \chi_{>k}(r), \qquad |c_k| \lesssim 1,
\]
where $ \chi_{>k}$ are  bump functions which equal $0$ for
$r \ll 2^m$, equal $1$ for $ r \gg 2^m$ and are smooth on the 
$2^m$ scale.We denote by $e_k$ the contributions of each summand to $e$ 
and consider three cases depending on the size of $r$:

i) $ r \ll 2^{k}$. Then $e_k$ has the form
\[
 e_k(r,\tau) = c_k \tau^2 \chi_{<2^{-k}}(|\tau|)
\]
therefore its corresponding output in $u_d$ 
satisfies 
\[
 |u_d^k| + 2^{k} |\partial_t u_d^k|\lesssim 2^{-3k} (1+ 2^{k} t)^{-N}
\]
which leads to \eqref{pd} and \eqref{pdt} after summation
with respect to $k$ with $2^k > r$.

ii) $r \gg 2^k$. Then $e_k$ has the form
\[
 e_k(r,\tau) = c_k \tau \la r \ra^{-1} \chi_{< 2^{-k}}(|\tau|)
\]
therefore its corresponding output in $u_d$ 
satisfies 
\[
 |u_d^k| + 2^{k} |\partial_t u_d^k|\lesssim \la r \ra^{-1} 2^{-2k} (1+ 2^{k} t)^{-N}
\]
which leads again to \eqref{pd} and \eqref{pdt} since now we are 
summing  with respect to $k$ for $2^k < r$.

iii) $r \approx 2^k$. Here we get terms as in both (i) and (ii),
but there is no summation with respect to $k$.

The proof of Theorem~\ref{maint} is concluded.

\bibliography{gr}
\bibliographystyle{plain}

\end{document}